\documentclass[11pt]{amsart}
\usepackage{amsfonts}
\usepackage{amsmath}
\usepackage{amssymb}
\usepackage{url}
\usepackage{multicol}
\usepackage{stmaryrd}
\usepackage{cite}
\usepackage{epsfig}
\usepackage{color}
\usepackage{graphics}
\usepackage{graphicx}
\usepackage{epstopdf}
\usepackage{multicol,graphics}
\newcommand\bes{\begin{eqnarray}}
\newcommand\ees{\end{eqnarray}}

\newtheorem{theorem}{Theorem}[section]
\newtheorem{lemma}[theorem]{Lemma}

\newtheorem{definition}[theorem]{Definition}
\newtheorem{remark}[theorem]{Remark}
\newtheorem{proposition}[theorem]{Proposition}
\numberwithin{equation}{section}
\allowdisplaybreaks

\begin{document}
\title[Asymptotic Behaviors for Nonlocal Diffusion Equations]{\textbf{Asymptotic Behaviors for Nonlocal Diffusion Equations about the Dispersal Spread}}

\author[Y.H. Su, W.T. Li and F.Y. Yang]{Yuan-Hang Su$^{1}$, Wan-Tong Li$^{1,*}$ and Fei-Ying Yang$^{1}$}
\thanks{\hspace{-.6cm}
$^1$ School of Mathematics and Statistics, Lanzhou University, Lanzhou, Gansu, 730000, People's Republic of China.
\\
$^*${\sf Corresponding author} (wtli@lzu.edu.cn)}

\date{\today}

\begin{abstract}
This paper studies the effects of the dispersal spread, which characterizes the dispersal range, on nonlocal diffusion equations with the nonlocal dispersal operator $ \frac{1}{\sigma^{m}}\int_{\Omega}J_{\sigma}(x-y)(u(y,t)-u(x,t))dy$ and Neumann boundary condition in the spatial heterogeneity environment. More precisely, we are mainly concerned with asymptotic behaviors of generalised principal eigenvalue to the nonlocal dispersal operator, positive stationary solutions and solutions to the nonlocal diffusion KPP equation in both large and small dispersal spread. For large dispersal spread, we show that their asymptotic behaviors are unitary with respect to the cost parameter $m\in[0,\infty)$.
However, small dispersal spread can lead to different asymptotic behaviors as the cost parameter $m$ is in a different range. In particular, for the case $m=0$, we should point out that asymptotic properties for the nonlocal diffusion equation with Neumann boundary condition are different from those for the nonlocal diffusion equation with Dirichlet boundary condition.


\end{abstract}

\maketitle


\section{Introduction}
\noindent

In this paper, we are interested in the following nonlocal diffusion equation
\begin{equation}  \label{101}
u_{t}(x,t)=\frac{1}{\sigma^{m}}\int_{\Omega}J_{\sigma}(x-y)(u(y,t)-u(x,t))
dy+f(x,u(x,t)),\ \  (x,t)\in \Omega\times(0,\infty),
\end{equation}
where $\Omega\subset\mathbb{R}^N$ is a bounded smooth domain, $u(x,t)$ represents the population density at location $x$ and time $t$, the dispersal spread $\sigma>0$ characterizes the dispersal range, the cost
parameter $m\in[0,\infty)$,
$J_{\sigma}(\cdot)=\frac{1}{\sigma^{N}}J(\frac{\cdot}{\sigma})$ is the scaled dispersal kernel. The nonlocal diffusion is described by
\begin{equation*}
M_{\sigma,m,\Omega}[u](x,t):=
\frac{1}{\sigma^{m}}\int_{\Omega}J_{\sigma}(x-y)(u(y,t)-u(x,t))dy.
\end{equation*}
Since we only integrate over $\Omega$, we assume that diffusion takes place only in $\Omega$. The individuals may not enter or leave the domain, which is called \textbf{nonlocal Neumann boundary condition}, see Andreu-Vaillo et al. \cite{Andreu-2010-AMS} and Cort\'{a}zar et al. \cite{Rossi-2007-JDE}.
Throughout this paper, we will always make the following assumptions on the dispersal kernel $J$ and the nonlinear function $f$:
\begin{itemize}
\item[(J)] $J\in C(\mathbb{R}^N)$ is nonnegative symmetric with
compact support on the unit ball $B_{1}(0)$, $J(0)>0$ and $\int_{\mathbb{R}^N}J(z)dz=1$.
\item[(F)] $f\in C^{0,1}(\bar{\Omega}\times\mathbb{R})$ is of KPP type,
that is:
\begin{eqnarray*}
\begin{cases}
f(\cdot,0)\equiv 0,\\
\text{For all} \ x\in \bar{\Omega}, f(x,s)/s \ \text{is decreasing with respect
to} \ s  \ \text{on} \ (0,+\infty),\\
\text{There exists} \ M>0 \ \text{such that} \ f(x,s)\leq 0 \ \text{for all} \
s\geq M \ \text{and all} \ x\in \bar{\Omega}.\\
\end{cases}
\end{eqnarray*}
\end{itemize}
A typical example of such a nonlinearity is given by $f(x,s):=s(a(x)-s)$ with $a\in C(\bar{\Omega})$.

Nonlocal dispersal equations of the form \eqref{101} are widely used to model dispersal phenomena which exhibit nonlocal internal interactions and have attracted much attention, see Bates and Zhao \cite{Bates-2007-JMAA}, Berestycki et al. \cite{Berestycki-2016-JMB}, Cao et al. \cite{CDLL-2018}, Chasseigne et al. \cite{Rossi-2006-JMPA}, Cort\'{a}zar et al. \cite{Rossi-2007-JDE}, Coville \cite{Coville-2010-JDE,Coville-2015-DCDS}, Coville et al. \cite{Coville-2013-Poincar}, Fife \cite{Fife-2003-Trends}, Hutson et al. \cite{Hutson-2003-JMB}, Kao et al. \cite{Kao-2010-DCDS}, Li et al. \cite{LSW2010,LWZ2016,LWZJNS2018,LZZ2015}, Shen and Xie \cite{Shen-2015-JDE}, Shen and Zhang \cite{Shen-2010-JDE}, Su et al. \cite{Su-2018-submitted}, Sun et al. \cite{SLW2011,SZLW2019}, Wang and Lv \cite{WMX}, Yang et al. \cite{Yang-2016-DCDSA} and Zhang et al. \cite{ZLW2017,ZLW2016,ZLWS2019}. From both mathematical and ecological points of view, nonlocal dispersal kernel functions can have a variety of forms. Here, we take the method given by Hutson et al. \cite{Hutson-2003-JMB} which the nonlocal dispersal kernel function can be quantitatively described by the dispersal rate $\mu$ and the dispersal spread $\sigma$. Moreover, they introduced the concept of a dispersal budget and showed that the dispersal rate is characterized by $\frac{\mu}{\sigma^{m}}$ under some conditions. From an evolutionary point of view, the species can ``choose'' to disperse a few offsprings over a long distance or many offsprings over a short distance or some compromise.
Recently, for $0\le m\le2$, Berestycki et al. \cite{Berestycki-2016-JFA, Berestycki-2016-JMB} have considered the asymptotic limits of generalised principal eigenvalue associated with the nonlocal dispersal operator $\frac{1}{\sigma^m}
\int_{\Omega}J_{\sigma}(x-y)(\varphi(y)-\varphi(x))dy$, and positive stationary solutions corresponding to the nonlocal diffusion equation. For a continuation of the works of Berestycki et al. \cite{Berestycki-2016-JFA, Berestycki-2016-JMB}, Shen and Vo \cite{shen-2019} studied the corresponding problems for the nonlocal dispersal equation with Dirichlet boundary condition in time-periodic media. Here, we study asymptotic behaviors for the nonlocal dispersal operator/equation with Neumann boundary condition.

This paper is devoted to discussing asymptotic behaviors of generalised principal eigenvalue to the nonlocal dispersal operator, positive stationary solutions and solutions to the nonlocal diffusion KPP equation with respect to the dispersal spread. Firstly, we consider the generalised principal eigenvalue to the nonlocal dispersal operator. It is known from \cite{Coville-2010-JDE,Donsker-1975-PNAS,Kao-2010-DCDS,Shen-2010-JDE} that
the operator $M_{\sigma ,m,\Omega }+a$ may not have any eigenvalues in the space $L^{p}(\Omega )$ or $C(\overline{\Omega })$.
Recently, there are some works on the nonlocal principal eigenvalue problems, see Berestycki et al. \cite{Berestycki-2016-JFA}, Coville \cite{Coville-2010-JDE}, Coville et al. \cite{Coville-2013-Poincar},
Diaconis et al. \cite{Diaconis-2011-IM},
Garc{\'{\i}}a-Meli{\'a}n and Rossi \cite{Rossi-2009-JDE}, Li et al. \cite{Li-2017-DCDS}, Liang et al. \cite{liang2017}, Shen and Xie \cite{Shen-2015-DCDSB}, Shen and Zhang \cite{Shen-2010-JDE}, Sun et al. \cite{Sun-2015-DCDS,SLY2014} and Yang et al. \cite{Yang-2016-DCDSA}. Their results imply that the generalised principal eigenvalue $\lambda_{p}(M_{\sigma,m,\Omega}+a)$:
\begin{align*}
\lambda_{p}(M_{\sigma,m,\Omega}+a):=\sup\big\{&\lambda\in \mathbb{R}\ |\ \exists \varphi\in C(\overline{\Omega}),\varphi>0, \ \text{such \ that }\\
&\ \ \ \ \ \ \ \ \ M_{\sigma,m,\Omega}[\varphi]+(a(x)+\lambda)\varphi\leq0\ \ \text{in}\ \Omega \big\},
\end{align*}
can be used as surrogates of a principal eigenvalue.
However, it is still hard to study asymptotic behaviors of $\lambda_{p}(M_{\sigma,m,\Omega}+a)$ about $\sigma$ since nonlocal
dispersal operators have not smoothness and certain compactness properties. In order to overcome these difficulties, we use a
combination of some estimations in \cite{Berestycki-2016-JFA,Coville-2010-JDE,Shen-2015-DCDSB},
new characterisation of Sobolev spaces in \cite{Bourgain-2001,Brezis-2001,Ponce-2004-JEMS} and some properties
of nonlocal dispersal operators with Neumann boundary condition to obtain asymptotic behaviors of $\lambda_{p}(M_{\sigma,m,\Omega}+a)$.
Secondly, we study positive stationary solutions to the nonlocal diffusion KPP equation.
An optimal persistence criterion about equation (\ref{101}) is obtained by the sign of the generalised principal eigenvalue, see Bates and Zhao \cite{Bates-2007-JMAA} and Coville \cite{Coville-2010-JDE}. Moreover, we analyze the effects of the dispersal spread and the dispersal budget on the positive stationary solutions. Owing to the lack of regularity of
the stationary solutions, we can not rely on standard compactness results
but take advantage of both new compactness results of Sobolev spaces in \cite{Bourgain-2001,Ponce-2004-JEMS} and some properties of nonlocal dispersal equations with Neumann boundary condition to get asymptotic behaviors of the positive stationary solutions. Finally, for the initial-boundary value problem of \eqref{101}, many works focus on small dispersal spread as $m=2$, see Andreu-Vaillo et al. \cite{Andreu-2010-AMS}, Shen and Xie \cite{Shen-2015-JDE} and reference therein. Here, we are based on upper-lower solutions
method to consider asymptotic behaviors of solutions to the nonlocal diffusion KPP equation \eqref{101} with respect to large dispersal spread and small dispersal spread. Particularly, the precise convergence rate is given in this paper.

Let us now state our main results. Firstly, we investigate the properties of $\lambda_p(M_{\sigma,m,\Omega}+a)$ with respect to the dispersal spread $\sigma$. More precisely, we get the continuous dependence of $\lambda_p(M_{\sigma,m,\Omega}+a)$ with respect to $\sigma$ and the limits of $\lambda_p(M_{\sigma,m,\Omega}+a)$ when $\sigma$ tends to zero or $+\infty$. It should be pointed out that the asymptotic limits of $\lambda_p(M_{\sigma,m,\Omega}+a)$ for $m=0$ can be discussed by the same method in \cite{Shen-2015-DCDSB}. Here,
we have the following main result.

\begin{theorem}
\label{th101} Assume that $J$ satisfies $(J)$, $a\in C(\bar{\Omega})$ and $m\in[0,\infty)$. Then $%
\lambda_p(M_{\sigma,m,\Omega}+a)$ is continuous with respect to $\sigma$.
Moreover, we have
\begin{equation*}
\lim\limits_{\sigma\to+\infty}
\lambda_p(M_{\sigma,m,\Omega}+a) =-\sup\limits_{\Omega}a.
\end{equation*}
\end{theorem}

\begin{remark}\label{re101}{\rm
(i) If $\lambda_p(M_{\sigma,m,\Omega}+a)$ is an isolated principal eigenvalue,
then $\lambda_p(M_{\sigma,m,\Omega}+a)$ is naturally continuous with respect
to the parameter $\sigma$, see Kato \cite{Kato-1995}. However, it is not clear for the generalised principal eigenvalue $\lambda_p(M_{\sigma,m,\Omega}+a)$. (ii) For the case $m=0$, Berestycki et al. \cite{Berestycki-2016-JFA} obtained the following result
\begin{equation*}
\lim\limits_{\sigma\to+\infty}
\lambda_p(M^{D}_{\sigma,m,\Omega}+a) =1-\sup\limits_{\Omega}a,
\end{equation*}
where nonlocal dispersal operator with Dirichlet boundary $M^{D}_{\sigma,m,\Omega}$ is
\begin{equation*}
M^{D}_{\sigma,m,\Omega}[\varphi](x):=\frac{1}{\sigma^{m}}
\bigg(\int_{\Omega}J_{\sigma}(x-y)\varphi(y)dy-\varphi(x)\bigg).
\end{equation*}
For Neumann boundary condition, we prove that
\begin{equation*}
\lim\limits_{\sigma\to+\infty}
\lambda_p(M_{\sigma,m,\Omega}+a) =-\sup\limits_{\Omega}a,
\end{equation*}
which is different from that of \cite{Berestycki-2016-JFA}.}
\end{remark}

It is an interesting question when the generalised principal eigenvalue is the principal eigenvalue. For the small dispersal spread, we prove that the generalised principal eigenvalue is really the principal eigenvalue in this paper. Moreover, the asymptotic limits of the principal eigenpair are also obtained.
Let us denote the second moments of $J$ by
$
D_{2}(J):=\int_{\mathbb{R}^{N}}J(z)|z|^{2}dz.
$
Then we have the following results.

\begin{theorem}
\label{th102} Assume that $J$ satisfies $(J)$ and $a\in C(\bar{\Omega})$. Then there exists $\sigma_{0}>0$ such that there is a positive continuous eigenfunction $\varphi_{p,\sigma}$ associated to $\lambda_p(M_{\sigma,m,\Omega}+a)$ with $\|\varphi_{p,\sigma}\|_{L^{2}(\Omega)}=1$ for all $0<\sigma\leq \sigma_{0}$.
Moreover, the asymptotic behaviors of principal eigenpair are divided into
the following three cases:
\begin{itemize}
\item[(i)] For $0\leq m<2$, we have
$
\lim\limits_{\sigma\to0^{+}}\lambda_p(M_{\sigma,m,\Omega}+a)=
-\sup\limits_{\Omega}a;
$
\item[(ii)] For $m=2$, assume further that $J$ is radially symmetric, $p_{\sigma}(x):=\int_{\Omega}J_{\sigma}(x-y)dy\in
C^{0,\alpha_{1}}(\overline{\Omega})$ with some $\alpha_{1}>0$ and $a\in C^{0,\alpha_{2}}(\overline{\Omega})$ with some $\alpha_{2}>0$. Then we have
$
\lim\limits_{\sigma\to0^{+}}\lambda_p(M_{\sigma,m,\Omega}+a)=
\lambda^{N}_1\Big(\frac{D_{2}(J)}{2N}\Delta+a\Big),
\text{ and }
\underset{\sigma\rightarrow 0^{+}}{\lim}\|\varphi_{p,\sigma}-\varphi_1\|_{L^{2}(\Omega)}=0,
$
where $(\lambda^{N}_{1}, \varphi_{1})$ with $\|\varphi_{1}\|_{L^{2}(\Omega)}=1$ is the principal eigenpair of
the following random dispersal operator
\begin{equation*}
\begin{cases}
\frac{D_{2}(J)}{2N}\Delta\varphi+(a(x)+\lambda)\varphi=0 & \text{ in }~
\Omega, \\
\frac{\partial\varphi}{\partial\nu}=0 & \text{ on }~\partial\Omega.%
\end{cases}%
\end{equation*}
Here, $\nu$ is the unit outward normal vector on $\partial\Omega$;
\item[(iii)] For $m>2$, if $J$ is radially symmetric, then we have
\begin{equation*}
\lim\limits_{\sigma\to0^{+}}\lambda_p(M_{\sigma,m,\Omega}+a)=
-\bar{a},
\end{equation*}
\begin{equation*}
\underset{\sigma\rightarrow 0^{+}}{\lim}\|\varphi_{p,\sigma}-
|\Omega|^{-\frac{1}{2}}\|_{L^{2}(\Omega)}=0,
\end{equation*}
where $\bar{a}:=\frac{1}{|\Omega|}\int_{\Omega}a(x)dx$.
\end{itemize}
\end{theorem}

\begin{remark}\label{re102}{\rm
Denote $\lambda_p(M^{D}_{\sigma,m,\Omega}+a)$ the generalised principal eigenvalue corresponding to the operator $M_{\sigma,m,\Omega}+a$ with Dirichlet boundary condition.
When $m>2$,
we conjecture that
\begin{equation*}
\lim\limits_{\sigma\to0^{+}}\lambda_p(M^{D}_{\sigma,m,\Omega}+a)=
\infty,
\end{equation*}
which is not considered by Berestycki et al. \cite{Berestycki-2016-JFA}.
}
\end{remark}

Next, we intend to understand the effects of the dispersal spread and the
dispersal budget on the persistence of the population. Before discussing
these issues, we recall an useful result \cite{Bates-2007-JMAA,Coville-2010-JDE}.
\begin{lemma}
\label{th207} Assume that $(J)$ and $(F)$ hold. Then there exists a unique positive continuous stationary solution $\theta _{\sigma }$ of \eqref{101} if and only if
$\lambda _{p}(M_{\sigma,m,\Omega }+a)<0$.
\end{lemma}

Based on this conclusion and Theorems \ref{th101} and \ref{th102}, we obtain the existence and uniqueness of positive stationary solutions to equation \eqref{101} when $\sigma $ is enough small or large. Furthermore, we analyze asymptotic limits of the positive stationary solutions as $\sigma$ tends to zero or $+\infty$. As explained in \cite{Hutson-2003-JMB,shen-2019}, these asymptotic as $\sigma \ll 1$ or $\sigma \gg 1$ represent two completely different dispersal strategies. On the one case, the limit as $\sigma\rightarrow 0^{+}$ can be associated to a strategy of dispersing many offspring on a short range. On the other case, the limit as $\sigma \rightarrow +\infty $ corresponds to a strategy that disperses a few offspring over a long distance.

In the present paper, our analysis will be divided into two distinct situations: $\sigma \gg 1$ and $\sigma \ll 1$. To simplify our presentation, we restrict our discussion to nonlinearity $f(x,s)=s(a(x)-s)$ with $a\in C(\bar{\Omega})$ and $a^{+}\not\equiv 0$. However, most of the proofs apply to a more general nonlinearity
$f(x,s)$ and $a(x)=f_{s}(x,0)$. We stress that $a^{+}\not\equiv 0$ is necessary for the existence of positive solutions. Indeed, if $a^{+}\equiv 0$, then for any positive constant $C_{0}$
, we have
\begin{equation*}
M_{\sigma,m,\Omega}[C_{0}]+a(x)C_{0}\leq 0.
\end{equation*}
Therefore, $\lambda_{p}(M_{\sigma,m,\Omega}+a)\geq0$ and there is no positive stationary solutions of equation \eqref{101} for all $\sigma$.
Our first result deals with the case $\sigma \gg 1$.

\begin{theorem}
\label{th103} Assume that $J$ satisfies $(J)$, $a\in C(\bar{\Omega})$ with $a^{+}\not\equiv 0$ and $m\in[0,\infty)$. Then there exists a positive stationary solution $\theta _{\sigma }$ to \eqref{101} for all $\sigma \gg 1$.
Moreover, we have%
\begin{equation*}
\underset{\sigma \rightarrow +\infty }{\lim }||\theta _{\sigma
}-a^{+}||_{L^{\infty }(\Omega )}=0.
\end{equation*}
\end{theorem}

\begin{remark}\label{re103}{\rm
For $m=0$, Berestycki et al. \cite{Berestycki-2016-JMB} showed that positive stationary solutions of the following equation
\begin{equation}  \label{102}
u_{t}(x,t)=\frac{1}{\sigma^{m}}\bigg(\int_{\Omega}J_{\sigma}(x-y)u(y,t)
dy-u(x,t)\bigg)+f(x,u(x,t)),\ \  (x,t)\in \Omega\times(0,\infty),
\end{equation}
may not exist for all $\sigma\gg 1$. However, positive stationary solutions of equation \eqref{101} exists for $\sigma\gg 1$. This seems to show that the nonlocal diffusion equation \eqref{101} is a conservative ecological system.}
\end{remark}

From the ecological point of view, the large spread strategy can be selected for species to persist as the case $m\geq 0$. That may be because organisms is concentrated on large dispersal range, but they ignores quantity of diffusion, which amounts to the small dispersal rate. However, the small spread strategy may not be an optimal strategy, in the sense that a population adopting such strategy can go extinct. When $m\geq 2$, it may happen that no positive solution exists for small $\sigma $. Here is our  precise result.

\begin{theorem}
\label{th104} Assume that $J$ satisfies $(J)$ and $a\in C(\bar{\Omega})$ and $a^{+}\not\equiv 0$. Then we obtain the following results.
\begin{itemize}
\item[(i)] For $0\leq m<2$, there exists a positive stationary solution $\theta _{\sigma }$ to \eqref{101} for all $0<\sigma \ll 1$. Assuming further that $a\in C^{2}(\overline{\Omega })$ and $J$ is radially symmetric, we have
\begin{equation*}
\underset{\sigma \rightarrow 0^{+} }{\lim }\ \|\theta _{\sigma }-V_{1}\|_{L^{1}(\Omega)}=0,
\end{equation*}
where $V_{1}$ is a nonnegative bounded solution of the following equation
\begin{equation*}
V_{1}(a(x)-V_{1})=0\text{ \ in }~\Omega \ ;
\end{equation*}
\item[(ii)] For $m=2$, suppose that the assumptions in Theorem \ref{th102} (ii) hold and $\lambda^{N}_{1}\Big(\frac{D_{2}(J)}{2N}\Delta +a\Big)<0$. Then
    there exists a positive stationary solution $\theta _{\sigma }$ to \eqref{101} for all $0<\sigma \ll 1$. Moreover, we obtain
\begin{equation*}
\underset{\sigma \rightarrow 0^{+} }{\lim }\ \|\theta _{\sigma }-V_{2}\|_{L^{2}(\Omega)}=0,
\end{equation*}
where $V_{2}$ is the unique bounded non-trivial solution of
\begin{align*}
\begin{cases}
\frac{D_{2}(J)}{2N}\Delta V_{2}+V_{2}(a(x)-V_{2}) =0\text{ \ \ \ in }\Omega \text{,} \\
\frac{\partial V_{2}}{\partial \nu} =0\text{ \ \ \ \ \ \ \ \ \ \ \ \ \ \ \ \ \ \ \
\ \ \ \ \ \ \ \ \ \ on }\partial \Omega;
\end{cases}
\end{align*}
\item[(iii)] For $m>2$, if $J$ is radially symmetric and $\bar{a}>0$, then there exists a positive stationary solution $\theta _{\sigma }$ to \eqref{101} for all $0<\sigma \ll 1$. Moreover, we get
\begin{equation*}
\underset{\sigma \rightarrow 0^{+} }{\lim }\ \|\theta _{\sigma }-\bar{a}\|_{L^{2}(\Omega)}=0.
\end{equation*}
\end{itemize}
\end{theorem}

\begin{remark}\label{re104}{\rm
(i) For $m=2$, if $\lambda^{N}_{1}\Big(\frac{D_{2}(J)}{2N}\Delta +a\Big)>0$, then it is easy to verify that there
exists no positive stationary solutions of equation \eqref{101}
for all $0<\sigma\ll 1$. However, it is an open problem for the critical case $\lambda^{N}_{1}\Big(\frac{D_{2}(J)}{2N}\Delta +a\Big)=0$. (ii) For $m>2$, if the conjecture of Remark \ref{re102}
is right, then there exists no positive stationary solutions of
equation (\ref{102}) for all $0<\sigma\ll 1$. However, when $\bar{a}>0$, a positive stationary solution of
equation \eqref{101} exists for all $0<\sigma\ll 1$; when $\bar{a}<0$,
there exists no positive stationary solutions of
equation \eqref{101} for all $0<\sigma\ll 1$; when $\bar{a}=0$, it
is an open problem. This shows that Dirichlet and Neumann boundary conditions have a great difference. Our result also shows that the spatial homogeneity can be caused by small dispersal spread in the spatial heterogeneous environment.}
\end{remark}

These results clearly highlight the dependence of the spreading strategy on the cost parameter $m$ and the distribution of the intrinsic
growth rate $a(x)$. Especially, we emphasize that the boundary condition may play an important role in the persistence of populations. For example, the large spread strategy with Neumann boundary condition can be selected for species to persist while one with Dirichlet boundary condition \cite{Berestycki-2016-JMB} may not be this case for $m=0$.

Finally, let us consider solution $u_{\sigma}(x,t)$ of equation \eqref{101} with initial value $u_0(x)$. We point out that the existence and uniqueness of solution $u_{\sigma}(x,t)$ has been studied by \cite{Bates-2007-JMAA,
Rossi-2007-JDE}, see also Proposition \ref{pr401} in the present paper. Specifically, we analyze the asymptotic behaviors of solution $u_{\sigma}(x,t)$ with respect to large dispersal spread and small dispersal spread.
Now, we obtain the following result for $m\in [0,\infty)$.
\begin{theorem}
\label{th106} Assume that $J$ satisfies $(J)$, $a\in C(\bar{\Omega})$ and $u_{0}\in C(\bar{\Omega})$ with $u_{0}\geq 0$. Then, for every $T\in (0,\infty)$, there exist $\sigma_{1}>0$ and $C(T)>0$ such that
\begin{equation*}
\sup\limits_{t\in[0,T]}\|u_{\sigma}(\cdot,t)-v(\cdot,t)\|_{L^{\infty}
(\Omega)}\leq C(T)\sigma^{-(m+N)},
\end{equation*}
for all $\sigma\geq \sigma_{1}$, where $v(x,t)$ satisfies the following
equation
\begin{equation}\label{106}
\begin{cases}
v_{t}(x,t)=v(a(x)-v), \ \ &(x,t)\in \bar{\Omega}\times(0,\infty),\\
v(x,0)=u_{0}(x),  \ \ &x\in \bar{\Omega}.
\end{cases}
\end{equation}
\end{theorem}

For small dispersal spread, we discuss asymptotic behavior of solution $u_{\sigma}(x,t)$ and give the precise convergence rate.
\begin{theorem}
\label{th107} Let $0\leq m<2$. Assume that $J$ satisfies $(J)$ and is radially symmetric, $a\in C^{2}(\bar{\Omega})$ and $u_{0}\in C^{2}(\bar{\Omega})$ with $u_{0}\geq 0$. Then, for every $T\in (0,\infty)$, there exist $\sigma_{0}>0$ and $C(T)>0$ such that
\begin{equation*}
\sup\limits_{t\in[0,T]}\|u_{\sigma}(\cdot,t)-v(\cdot,t)\|_{L^{\infty}
(\Omega)}\leq C(T)\sigma^{2-m},
\end{equation*}
for all $0<\sigma\leq \sigma_{0}$. Here, $v(x,t)$ is the solution of \eqref{106}.
\end{theorem}

\begin{remark}\label{re105}{\rm
(i) For $m=2$, some results have been given by \cite{Andreu-2010-AMS,Shen-2015-JDE}. (ii) For $m>2$, we conjecture that $u_{\sigma}(x,t)$ can approximate
to the solution of the following equation
\begin{equation*}
\begin{cases}
v_{t}(t)=v(\bar{a}-v), \ \ &(x,t)\in \bar{\Omega}\times(0,\infty),\\
v(x,0)=u_{0}(x),  \ \ &x\in \bar{\Omega}.
\end{cases}
\end{equation*}
}
\end{remark}

In short, the present paper deals with both large and small dispersal spread strategies. For large dispersal spread strategy, we show that their asymptotic are unitary with respect to the cost parameter $m$. However, small dispersal spread strategy can lead to different asymptotic behaviors as the cost parameter $m$ is in a different range. For instance, stationary solutions of the nonlocal diffusion KPP equation can approximate to stationary solutions of the corresponding kinetic equation, the corresponding random diffusion KPP equation or the corresponding spatial homogeneous kinetic equation. The intermediate dispersal spread strategy, which may be investigated by a bifurcation approach in our future work, is also of great biological interest.

The rest of the paper is organised as follows. In Section 2, we first establish the continuous dependence of $\lambda _{p}(M_{\sigma ,m,\Omega}+a)$ with respect to the parameter $\sigma $. Then we study the asymptotic limits of $\lambda _{p}(M_{\sigma ,m,\Omega}+a)$ and the corresponding principal eigenfunction $\varphi_{p,\sigma}$ (i.e. Theorems \ref{th101} and \ref{th102}). Section 3 is devoted to investigating the effects of the dispersal spread and the dispersal budget on persistence criteria (i.e. Theorems \ref{th103} and \ref{th104}). The last section concerns
the asymptotic behaviors of solution $u_{\sigma}(x,t)$ to equation \eqref{101} for large dispersal spread and small dispersal spread (i.e. Theorems \ref{th106} and \ref{th107}).


\section{Asymptotic limits of principal eigenpair}

\noindent

In this section, we investigate the following spectral problem
\begin{equation}\label{301}
M_{\sigma ,m,\Omega }[\varphi ]+(a(x)+\lambda )\varphi =0\text{ \ \ \ in } \Omega,
\end{equation}
where $\Omega\subset\mathbb{R}^N$ is a bounded smooth domain. As noticed in \cite{Coville-2010-JDE,Donsker-1975-PNAS,Kao-2010-DCDS,Shen-2010-JDE}, the operator $M_{\sigma ,m,\Omega }+a$ may not have any eigenvalues in the space $L^{p}(\Omega )$ or $C(\overline{\Omega })$. However, it is enough to discuss the dynamic behavior of \eqref{101} by the definition of the generalised principal eignevalue $\lambda _{p}(M_{\sigma ,m,\Omega}+a) $. Here, we establish the continuous dependence and asymptotic properties of $\lambda _{p}(M_{\sigma ,m,\Omega}+a)$ with respect to the parameter $\sigma $. In addition, we discuss the existence and asymptotic behaviors of a positive continuous eigenfunction $\varphi _{p,\sigma }$ associated to the principal eigenvalue $\lambda_{p}(M_{\sigma ,m,\Omega }+a)$ as $\sigma \ll 1$.

\subsection{Large dispersal spread}

\noindent

This subsection is dedicated to proving Theorem \ref{th101}. Before going to the study of these limits, we
obtain the continuous dependence of $\lambda _{p}(M_{\sigma ,m,\Omega }+a) $
with respect to $\sigma $.

\begin{theorem}
\label{le301} Assume that $J$ satisfies $(J)$ and $a\in C(\bar{\Omega})$. Then $%
\lambda _{p}(M_{\sigma ,m,\Omega }+a)$ is continuous with respect to $\sigma
$.
\end{theorem}

\begin{proof}
Let us first denote some notations
\begin{equation*}
L_{\sigma ,m,\Omega }[\varphi](x):=\frac{1}{\sigma ^{m}}\int_{\Omega
}J_{\sigma }(x-y)\varphi(y)dy
\end{equation*}
and
\begin{equation*}
b(x):=a(x)-\frac{1}{\sigma ^{m}}\int_{\Omega
}J_{\sigma }(x-y)dy \ , \ \nu =\underset{\Omega }{\sup }\ b.
\end{equation*}
Observe that with these notations, we have $\lambda _{p}(M_{\sigma ,m,\Omega
}+a)=\lambda _{p}(L_{\sigma ,m,\Omega }+b).$

By the definition of $\nu $, there exists a sequence of points $%
\{x_{k}\}_{k\in
\mathbb{N}
}$ such that $x_{k}\in \Omega $ and $|b(x_{k})-\nu |<1/k$. From continuity
of $b$, for each $k$, there exists $\eta _{k}>0$ such that%
\begin{equation*}
B_{\eta _{k}}(x_{k})\subset \Omega \text{ \ and \ }\underset{B_{\eta
_{k}}(x_{k})}{\sup }|b-\nu |\leq 2/k.
\end{equation*}

Next, let us take a sequence of $\{\epsilon _{k}\}_{k\in
\mathbb{N}
}$ which converges to zero such that $0<\epsilon _{k}\leq \eta _{k}/2$. Let $%
\chi _{k}$ be the following sequence of cut off functions:%
\begin{equation*}
\chi _{k}(x)=\chi \Big(\frac{|x-x_{k}|}{\epsilon _{k}}\Big),
\end{equation*}%
where $\chi $ is a smooth function such that $0\leq \chi \leq 1,\chi (z)=0$
for $|z|\geq 2$ and $\chi (z)=1$ for $|z|\leq 1$.

Finally, let us consider the continuous function $b_{k}(\cdot )$, defined by
$b_{k}(x):=\sup \{b(x),(\nu -\inf_{\Omega}b)\chi _{k}(x)+\inf_{\Omega}b\}.$ So we have%
\begin{equation*}
b_{k}(x)=
\begin{cases}
b(x)\text{ \ \ \ \ \ \ for }x\in \Omega \backslash B_{2\epsilon _{k}}(x_{k})%
\text{,} \\
\nu \text{ \ \ \ \ \ \ \ \ \ \ for }x\in \Omega \cap B_{\epsilon _{k}}(x_{k})%
\end{cases}%
\end{equation*}%
and%
\begin{equation*}
||b-b_{k}||_{L^{\infty }(\Omega )}\leq \sup_{B_{\eta _{k}}(x_{k})}
|b -\nu|\rightarrow 0\text{ \ \ \ as }k\rightarrow \infty \text{.}
\end{equation*}

By construction, for every given constant $\epsilon >0$, there exists $%
k_{0}\in
\mathbb{N}
$ such that for all $k\geq k_{0}$, we get%
\begin{equation*}
||b-b_{k}||_{L^{\infty }(\Omega )}<\epsilon/3
\end{equation*}%
and%
\begin{equation*}
\underset{\Omega }{\sup }\ b_{k}(\cdot )=\nu \ , \ \frac{1}{\nu -b_{k}}\notin
L^{1}(\Omega ).
\end{equation*}%
Let us take $b_{\epsilon }(\cdot )=b_{k_{0}}(\cdot )$.

Thanks to \cite[Theorem 1.1]{Coville-2010-JDE}, $\lambda _{p}(L_{\sigma,m,\Omega }+b_{\epsilon })$ is a
simple isolated principal eigenvalue. Let%
\begin{equation*}
a_{\epsilon }(x)=b_{\epsilon }(x)+\frac{1}{\sigma ^{m}}\int_{\Omega
}J_{\sigma }(x-y)dy,
\end{equation*}%
then%
\begin{equation*}
||a-a_{\epsilon }||_{L^{\infty }(\Omega )}<\epsilon/3
\end{equation*}%
and%
\begin{equation*}
\lambda _{p}(M_{\sigma ,m,\Omega }+a_{\epsilon })=\lambda _{p}(L_{\sigma
,m,\Omega }+b_{\epsilon }).
\end{equation*}%
It follows from the classical perturbation theory of isolated eigenvalues \cite{Kato-1995}. In
fact, for each fixed $\sigma _{0}>0$, we can write $M_{\sigma ,m,\Omega
}+a_{\epsilon }$ as%
\begin{equation*}
M_{\sigma ,m,\Omega }+a_{\epsilon }=M_{\sigma _{0},m,\Omega }+a_{\epsilon
}+U_{\sigma ,\sigma _{0}}\text{,}
\end{equation*}%
where%
\begin{equation*}
U_{\sigma ,\sigma _{0}}[\varphi ](x)=\frac{1}{\sigma ^{m}}\int_{\Omega
}J_{\sigma }(x-y)(\varphi (y)-\varphi (x))dy-\frac{1}{\sigma _{0}^{m}}%
\int_{\Omega }J_{\sigma _{0}}(x-y)(\varphi (y)-\varphi (x))dy\text{.}
\end{equation*}%
Because $U_{\sigma ,\sigma _{0}}$ is a linear bounded operator and $U_{\sigma
,\sigma _{0}}\rightarrow 0$ \ in norm as $\sigma \rightarrow \sigma _{0}$,
there exists $\delta _{0}>0$ such that for all $|\sigma -\sigma _{0}|<\delta
_{0}$, we have%
\begin{equation*}
\big|\lambda _{p}(M_{\sigma ,m,\Omega }+a_{\epsilon })-\lambda _{p}(M_{\sigma
_{0},m,\Omega }+a_{\epsilon })\big|<\epsilon/3.
\end{equation*}

By \cite[Proposition 1.1($iii$)]{Coville-2010-JDE} , $\lambda _{p}(M_{\sigma ,m,\Omega }+a)$
is Lipschitz continuous with respect to $a$, i.e.,%
\begin{equation*}
\big|\lambda _{p}(M_{\sigma ,m,\Omega }+a_{\epsilon })-\lambda _{p}(M_{\sigma
,m,\Omega }+a)\big|\leq ||a-a_{\epsilon }||_{L^{\infty }(\Omega )}<
\epsilon/3.
\end{equation*}

In a word, for every given constant $\epsilon >0$, there exist $\delta
_{0}>0 $ and $a_{\epsilon }\in C(\overline{\Omega })$ such that for all $%
|\sigma -\sigma _{0}|<\delta _{0}$, we have%
\begin{align*}
&|\lambda _{p}(M_{\sigma ,m,\Omega }+a)-\lambda _{p}(M_{\sigma _{0},m,\Omega
}+a)| \\
\leq &|\lambda _{p}(M_{\sigma ,m,\Omega }+a)-\lambda _{p}(M_{\sigma
,m,\Omega }+a_{\epsilon })|+|\lambda _{p}(M_{\sigma ,m,\Omega }+a_{\epsilon
})-\lambda _{p}(M_{\sigma _{0},m,\Omega }+a_{\epsilon })| \\
&+|\lambda _{p}(M_{\sigma _{0},m,\Omega }+a_{\epsilon })-\lambda
_{p}(M_{\sigma _{0},m,\Omega }+a)| \\
<&\frac{\epsilon }{3}+\frac{\epsilon }{3}+\frac{\epsilon }{3}=\epsilon.
\end{align*}%
So $\lambda _{p}(M_{\sigma ,m,\Omega }+a)$ is continuous with respect to $%
\sigma $.
\end{proof}

Next, we consider the limit behavior of $\lambda _{p}(M_{\sigma ,m,\Omega}+a)$ as $\sigma\rightarrow +\infty$.
\begin{theorem}
\label{le304} Assume that $J$ satisfies $(J)$
and $a\in C(\bar{\Omega})$. Then
\begin{equation*}
\underset{\sigma \rightarrow +\infty }{\lim }\lambda _{p}(M_{\sigma ,m,\Omega
}+a)=-\underset{\Omega }{\sup }\ a.
\end{equation*}
\end{theorem}

\begin{proof}
By the definition of $\lambda _{p}(M_{\sigma ,m,\Omega }+a)$, we get
\begin{equation*}
\lambda _{p}(M_{\sigma ,m,\Omega }+a)\leq -%
\underset{x\in \Omega }{\sup }\bigg\{a(x)-\frac{1}{\sigma ^{m}}\int_{\Omega
}J_{\sigma }(x-y)dy\bigg\}.
\end{equation*}
On the other hand, by using the test function $(\varphi,\lambda)=(1,-\sup_{\Omega}a)$, we can easily check
that for any $\sigma>0$
\begin{equation*}
-\underset{\Omega }{\sup }\ a\leq \lambda _{p}(M_{\sigma ,m,\Omega }+a).
\end{equation*}
Thus, we have
\begin{equation*}
-\underset{\Omega }{\sup }\ a\leq \lambda _{p}(M_{\sigma ,m,\Omega }+a)\leq -%
\underset{x\in \Omega }{\sup }\bigg\{a(x)-\frac{1}{\sigma ^{m}}\int_{\Omega
}J_{\sigma }(x-y)dy\bigg\}.
\end{equation*}
Since the following inequality%
\begin{equation*}
-\underset{x\in \Omega }{\sup }\bigg\{a(x)-\frac{1}{\sigma ^{m}}\int_{\Omega
}J_{\sigma }(x-y)dy\bigg\}\leq -\underset{\Omega }{\sup }\ a+\frac{1}{\sigma
^{m}}\underset{x\in \Omega }{\sup }\bigg\{\int_{\Omega }J_{\sigma }(x-y)dy%
\bigg\},
\end{equation*}%
we have%
\begin{equation*}
-\underset{\Omega }{\sup }\ a\leq \lambda _{p}(M_{\sigma ,m,\Omega }+a)\leq -%
\underset{\Omega }{\sup }\ a+\frac{1}{\sigma ^{m}}\underset{x\in \Omega }{\sup
}\bigg\{\int_{\Omega }J_{\sigma }(x-y)dy\bigg\}.
\end{equation*}%
As $\sigma \rightarrow +\infty ,$%
\begin{align*}
\int_{\Omega }J_{\sigma }(x-y)dy &=\frac{1}{\sigma ^{N}}\int_{\Omega }J\Big(%
\frac{x-y}{\sigma }\Big)dy \\
&=\frac{1}{\sigma ^{N}}\int_{\Omega }J(-z)d(x+\sigma z) \\
&=\int_{\frac{\Omega -x}{\sigma }}J(z)dz\rightarrow 0\text{, \ \ }\forall
x\in \Omega ,
\end{align*}%
which implies that%
\begin{equation*}
\underset{\sigma \rightarrow +\infty }{\lim }\lambda _{p}(M_{\sigma ,m,\Omega
}+a)=-\underset{\Omega }{\sup }\ a.
\end{equation*}
\end{proof}

\subsection{Small dispersal spread}

\subsubsection{Existence of principal eigenpair}

\noindent

Firstly, let us introduce three definitions.

\begin{definition}
\label{def205} We define the following quantities:%
\begin{align*}
\lambda _{p}^{\prime }(M_{\sigma ,m,\Omega }+a):=\inf \big\{& \lambda \in
\mathbb{R}\ |\ \exists \varphi \geq 0,\varphi \in C(\Omega )\cap L^{\infty
}(\Omega )\text{, such that} \\
\text{ }& \ \ \ \ \ \ \ \ \ \ \ \ \ \ \ \ M_{\sigma ,m,\Omega }[\varphi ]+(a(x)+\lambda )\varphi \geq 0\text{
\ in }\Omega \big\},\\
\lambda _{p}^{\prime \prime }(M_{\sigma ,m,\Omega }+a):=\inf \big\{& \lambda
\in \mathbb{R}\ |\ \exists \varphi \geq 0,\varphi \in C_{c}(\Omega )\text{, such
that} \\
& \ \ \ \ \ \ \ \ \ \ \ \ \ \ \ \  M_{\sigma ,m,\Omega }[\varphi ]+(a(x)+\lambda )\varphi \geq 0\text{ \ in }%
\Omega \big\},
\end{align*}%
and
\begin{align*}
\lambda _{v}(M_{\sigma ,m,\Omega }+a): =&\underset{\varphi \in L^{2}(\Omega
),\varphi \not\equiv 0}{\inf }-\frac{\langle M_{\sigma ,m,\Omega }[\varphi
]+a\varphi ,\varphi \rangle }{\langle \varphi ,\varphi \rangle } \\
=&\underset{\varphi \in L^{2}(\Omega ),\varphi \not\equiv 0}{\inf }\frac{%
\frac{1}{2\sigma ^{m}}\int_{\Omega }\int_{\Omega }J_{\sigma }(x-y)(\varphi
(y)-\varphi (x))^{2}dydx-\int_{\Omega }a(x)\varphi ^{2}(x)dx}{||\varphi
||_{L^{2}(\Omega )}^{2}}
\end{align*}%
where $\langle \cdot ,\cdot \rangle $ denotes the standard scalar product in $%
L^{2}(\Omega )$.
\end{definition}

By the same discussion as the proof in \cite{Berestycki-2016-JFA}, we have the following result.

\begin{lemma}
\label{th206} Assume that $J$ satisfies $(J)$ and $a\in C(\bar{\Omega})$. Then
we have%
\begin{equation*}
\lambda _{p}(M_{\sigma ,m,\Omega }+a)=\lambda _{p}^{\prime }(M_{\sigma
,m,\Omega }+a)=\lambda _{p}^{\prime \prime }(M_{\sigma ,m,\Omega
}+a)=\lambda _{v}(M_{\sigma ,m,\Omega }+a).
\end{equation*}
\end{lemma}

Now, we recall an useful criterion \cite{Coville-2013-AML}
that guarantees the existence of a continuous principal eigenfunction.

\begin{lemma}
\label{th202}
Assume that $J$ satisfies $(J)$ and $a\in C(\bar{\Omega})$. Then there exists a positive continuous eigenfunction associated to $\lambda
_{p}(M_{\sigma ,m,\Omega }
+a)$ if and only if $\lambda
_{p}(M_{\sigma ,m,\Omega }+a)<-\sup_{x\in \Omega } \Big\{a(x)-\frac{1}{%
\sigma ^{m}}\int_{\Omega }J_{\sigma }(x-y)dy\Big\}$.
\end{lemma}

Next, we start with discussing the existence of a positive continuous eigenfunction $\varphi _{p,\sigma }$ associated to the principal eigenvalue $\lambda_{p}(M_{\sigma ,m,\Omega }+a)$ as $\sigma \ll 1$. For $m=0$, this is true, see \cite{Shen-2015-DCDSB}. Here, we only consider the case $m>0$.
\begin{theorem}
\label{le307} Let $m>0$. If $J$ satisfies $(J)$ and $a\in C(\bar{\Omega})$, then there exists $\sigma_{0}>0$ such that there is a positive continuous eigenfunction $\varphi_{p,\sigma}$ associated to $\lambda_p(M_{\sigma,m,\Omega}+a)$ for all $0<\sigma\leq \sigma_{0}$.
\end{theorem}

\begin{proof}
According to Definition \ref{def205} and Lemma \ref{th206}, we have
\begin{equation*}
\lambda _{p}(M_{\sigma ,m,\Omega }+a)
=\underset{\varphi \in L^{2}(\Omega ),\varphi \not\equiv 0}{\inf }\frac{%
\frac{1}{2\sigma ^{m}}\int_{\Omega }\int_{\Omega }J_{\sigma }(x-y)(\varphi
(y)-\varphi (x))^{2}dydx-\int_{\Omega }a(x)\varphi ^{2}(x)dx}{||\varphi
||_{L^{2}(\Omega )}^{2}},
\end{equation*}%
which implies that
\begin{equation*}
-\underset{\Omega }{\sup }\ a\leq \lambda _{p}(M_{\sigma ,m,\Omega }+a)\leq -\bar{a}.
\end{equation*}
For all $x\in \Omega $, we have%
\begin{align*}
\int_{\Omega }J_{\sigma }(x-y)dy &=\int_{\frac{\Omega -x}{\sigma }}J(z)dz \\
&=\int_{%
\mathbb{R}
^{N}}J(z)\chi _{\frac{\Omega -x}{\sigma }}(z)dz\rightarrow 1\text{ \ \ \ \
as }\sigma \rightarrow 0^{+}.
\end{align*}%
Thus, there exists $\sigma _{1}>0$ such that for all $\sigma \leq \sigma
_{1},$%
\begin{equation*}
\int_{\Omega }J_{\sigma }(x-y)dy\geq 1/2\text{ \ \ \ for all }x\in \Omega .
\end{equation*}%
It follows that%
\begin{equation*}
\underset{x\in \Omega }{\sup }\bigg\{a(x)-\frac{1}{\sigma ^{m}}\int_{\Omega
}J_{\sigma }(x-y)dy\bigg\}\leq \underset{\Omega }{\sup }\ a-\frac{1%
}{2\sigma ^{m}},
\end{equation*}%
that is
\begin{equation*}
-\underset{x\in \Omega }{\sup }\bigg\{a(x)-\frac{1}{\sigma ^{m}}\int_{\Omega }J_{\sigma}(x-y)dy\bigg\}\geq \frac{1}{2\sigma ^{m}}-\underset{\Omega }{\sup }\ a.
\end{equation*}
It is easy to see that there exists $\sigma_{2}>0$ such that for
all $\sigma\leq \sigma_{2}$,
\begin{equation*}
\frac{1}{2\sigma ^{m}}-\underset{\Omega }{\sup }\ a> -\bar{a}.
\end{equation*}
In conclusion, for all $0<\sigma \leq \sigma _{0}:=\min \{\sigma _{1},\sigma _{2}\},$%
\begin{equation*}
-\underset{x\in \Omega }{\sup }\bigg\{%
a(x)-\frac{1}{\sigma ^{m}}\int_{\Omega }J_{\sigma }(x-y)dy\bigg\}>\lambda _{p}(M_{\sigma ,m,\Omega }+a),
\end{equation*}%
which, by Lemma \ref{th202}, enforces the existence of a positive
continuous principal eigenfunction $\varphi _{p,\sigma }$ associated with $\lambda
_{p}(L_{\sigma ,m,\Omega }+a)$. This ends the proof.
\end{proof}

\subsubsection{Asymptotic limits of principal eigenpair for small dispersal spread}

\noindent

To simplify our presentation, let us introduce following notations:%
\begin{align*}
&\mathcal{A(\varphi )} \mathcal{:=}\frac{\int_{\Omega }a(x)\varphi ^{2}(x)dx%
}{||\varphi ||_{L^{2}(\Omega )}^{2}}\text{, \ \ \ }\mathcal{J(\varphi ):=}%
\frac{D_{2}(J)}{2N}\frac{\int_{\Omega }|\nabla \varphi |^{2}(x)dx}{||\varphi
||_{L^{2}(\Omega )}^{2}}, \\
&\mathcal{I}_{\sigma ,m}\mathcal{(\varphi )} \mathcal{:=}\frac{-\frac{1}{%
\sigma ^{m}}\int_{\Omega }\int_{\Omega }J_{\sigma }(x-y)(\varphi (y)-\varphi
(x))\varphi(x)dydx}{||\varphi ||_{L^{2}(\Omega )}^{2}}-\mathcal{A(\varphi )}.
\end{align*}
We are now in position to obtain the asymptotic limits of $\lambda
_{p}(M_{\sigma ,m,\Omega }+a)$ as $\sigma$ tends to zero. For simplicity, we analyze three distinct situations:
$0\leq m<2$, $m=2$ and $m>2$. For the case
$0\leq m<2$, the proof is similar to \cite[Claim 4.2]{Berestycki-2016-JFA}. Here, we omit it. Next, we give the limit of $\lambda_{p}(M_{\sigma,2,\Omega}+a)$ as $\sigma\rightarrow 0^{+}$.

\begin{theorem}
\label{le306} Let $m=2$ and assume that $J$ satisfies $(J)$.
Assume further that $p_{\sigma}(x):=\int_{\Omega}J_{\sigma}(x-y)dy\in
C^{0,\alpha_{1}}(\overline{\Omega})$ with some $\alpha_{1}>0$ and $a\in C^{0,\alpha_{2}}(\overline{%
\Omega})$ with some $\alpha_{2}>0$. If $J$ is radially symmetric, then
\begin{align*}
\underset{\sigma\rightarrow 0^{+}}{\lim}\lambda_{p}(M_{\sigma,2,\Omega}+a)=
\lambda^{N}_{1}\Big(\frac{D_{2}(J)}{2N}\Delta+a\Big),
\end{align*}
where
\begin{align*}
\lambda^{N}_{1}\Big(\frac{D_{2}(J)}{2N}\Delta+a\Big):=\underset
{\varphi\in
H^{1}(\Omega),\varphi \not\equiv0}{\inf}\left\{\frac{D_{2}(J)}{2N} \frac{%
\int_{\Omega}|\nabla\varphi|^{2}(x)dx}{||\varphi||^{2}_{L^{2}(\Omega)}} -%
\frac{\int_{\Omega}a(x)\varphi^{2}(x)dx}{||\varphi||^{2}_{L^{2}(\Omega)}}
\right\}
\end{align*}
\end{theorem}

\begin{proof}
Let $\rho_{\sigma }(z):=\frac{1}{\sigma ^{2}D_{2}(J)}J_{\sigma }(z)|z|^{2}$, then $\rho _{\sigma }$ is a radial symmetric continuous mollifier such that%
\begin{equation*}
\begin{cases}
\rho _{\sigma }\geq 0,\text{~~~~~~~~~~~~~~~~~~~~~~~ in } \mathbb{R}^{N}, \\
\int_{\mathbb{R}^{N}}\rho _{\sigma }(z)dz=1,\text{~~~~~~~~~~~ }\forall \sigma >0, \\
\underset{\sigma \rightarrow 0^{+}}{\lim }\int_{|z|\geq \delta }\rho _{\sigma
}(z)dz=0,\text{~~~~~~~~~~~ }\forall \delta >0.
\end{cases}
\end{equation*}
For any $\varphi \in H^{1}(\Omega )$, we have%
\begin{align*}
\mathcal{I}_{\sigma ,2}(\varphi ) =&\frac{1}{||\varphi ||_{L^{2}(\Omega
)}^{2}}\bigg(\frac{1}{2\sigma ^{2}}\int_{\Omega }\int_{\Omega }J_{\sigma
}(x-y)(\varphi (x)-\varphi (y))^{2}dydx\bigg)-\mathcal{A}(\varphi ) \\
=&\frac{1}{||\varphi ||_{L^{2}(\Omega )}^{2}}\bigg(\frac{D_{2}(J)}{2}%
\int_{\Omega }\int_{\Omega }\rho _{\sigma }(x-y)\frac{(\varphi (x)-\varphi
(y))^{2}}{|x-y|^{2}}dxdy\bigg)-\mathcal{A}(\varphi ).
\end{align*}%
It follows from the characterisation of Sobolev spaces in
\cite{Bourgain-2001} that%
\begin{equation*}
\underset{\sigma \rightarrow 0^{+}}{\lim }\int_{\Omega }\int_{\Omega }\rho
_{\sigma }(x-y)\frac{(\varphi (x)-\varphi (y))^{2}}{|x-y|^{2}}%
dxdy=K_{2,N}||\nabla \varphi ||_{L^{2}(\Omega )}^{2},
\end{equation*}%
where%
\begin{equation*}
K_{2,N}=\frac{1}{|S^{N-1}|}\int_{S^{N-1}}(s\cdot \mathbf{e}_{1})^{2}ds=\frac{1}{N}.
\end{equation*}%
Thus, for any $\varphi \in H^{1}(\Omega )$, we have%
\begin{equation*}
\underset{\sigma \rightarrow 0^{+}}{\lim \sup }~~ \lambda _{p}(M_{\sigma ,2,\Omega
}+a)\leq \underset{\sigma \rightarrow 0^{+}}{\lim }\mathcal{I}_{\sigma
,2}(\varphi )=\mathcal{J}(\varphi )-\mathcal{A}(\varphi ).
\end{equation*}%
By definition of $\lambda _{1}\Big(\frac{D_{2}(J)}{2N}\Delta +a\Big)$, it is
then standard to obtain%
\begin{equation*}
\underset{\sigma \rightarrow 0^{+}}{\lim \sup }~~ \lambda _{p}(M_{\sigma ,2,\Omega
}+a)\leq \lambda^{N}_{1}\Big(\frac{D_{2}(J)}{2N}\Delta +a\Big).
\end{equation*}

On the other hand, we need to establish the following inequality%
\begin{equation*}
\lambda^{N}_{1}\Big(\frac{D_{2}(J)}{2N}\Delta +a\Big)\leq \underset{\sigma
\rightarrow 0^{+}}{\lim \inf }~~ \lambda _{p}(M_{\sigma ,2,\Omega }+a).
\end{equation*}%
Observe that to obtain the above inequality, it is sufficient to prove that%
\begin{equation*}
\lambda^{N}_{1}\Big(\frac{D_{2}(J)}{2N}\Delta +a\Big)\leq \underset{\sigma
\rightarrow 0^{+}}{\lim \inf }~~ \lambda _{p}(M_{\sigma ,2,\Omega }+a)+3\delta
\text{ \ \ for all }\delta >0.
\end{equation*}

Let us fix $\delta >0$. Firstly, to obtain the above inequality, we construct sufficiently smooth test function $\varphi _{\sigma }$.  We claim that, for all $\sigma >0,$ there exists $\varphi _{\sigma
}\in C_{c}^{\infty }(\Omega )$ such that%
\begin{equation*}
M_{\sigma ,2,\Omega }[\varphi _{\sigma }](x)+(a(x)+\lambda _{p}(M_{\sigma
,2,\Omega }+a)+3\delta )\varphi _{\sigma }(x)\geq 0\text{ \ \ in }x\in
\Omega .
\end{equation*}%
Indeed, by Lemma \ref{th206}, we have $\lambda _{p}(M_{\sigma ,2,\Omega
}+a)=\lambda _{p}^{\prime \prime }(M_{\sigma ,2,\Omega }+a).$ Therefore, for
all $\sigma $, there exists $\psi _{\sigma }\in C_{c}(\Omega )$ such that%
\begin{equation*}
M_{\sigma ,2,\Omega }[\psi _{\sigma }](x)+(a(x)+\lambda _{p}(M_{\sigma
,2,\Omega }+a)+\delta )\psi _{\sigma }(x)\geq 0\text{ \ \ for all }x\in
\Omega .
\end{equation*}%
Now, let $\eta $ be a smooth mollifier of unit mass and with support in the
unit ball. Then we consider $\eta _{\tau }(z):=\frac{1}{\tau ^{N}}\eta (%
\frac{z}{\tau })$ for $\tau >0$ and define%
\begin{equation*}
\widetilde{\varphi }_{\sigma }:=\eta _{\tau }\ast \psi _{\sigma }.
\end{equation*}%
For $\tau $ small enough, say $\tau \leq \tau _{0},$ the function $%
\widetilde{\varphi }_{\sigma }\in C_{c}^{\infty }(\Omega ).$

Let $p_{\sigma }(z)=\int_{\Omega }J_{\sigma }(z-y)dy$, it deduces from some
simple computations that%
\begin{align}
& \frac{1}{\sigma ^{2}}\int_{\Omega }\eta _{\tau }(x-z)\int_{\Omega
}J_{\sigma }(z-y)\psi _{\sigma }(y)dydz  \notag \\
=& \frac{1}{\sigma ^{2}}\int_{\Omega }J_{\sigma }(x-y)\int_{\Omega }\eta
_{\tau }(y-z)\psi _{\sigma }(y)dzdy  \notag \\
=& \frac{1}{\sigma ^{2}}\int_{\Omega }J_{\sigma }(x-y)\widetilde{\varphi }%
_{\sigma }(y)dy  \label{602}
\end{align}%
and%
\begin{align}
& \frac{1}{\sigma ^{2}}\int_{\Omega }\eta _{\tau }(x-z)\int_{\Omega
}J_{\sigma }(z-y)\psi _{\sigma }(z)dydz  \notag \\
=& \frac{1}{\sigma ^{2}}\int_{\Omega }\eta _{\tau }(x-z)\psi _{\sigma
}(z)p_{\sigma }(z)dz  \notag \\
=& \frac{1}{\sigma ^{2}}\int_{\Omega }\eta _{\tau }(x-z)\psi _{\sigma
}(z)p_{\sigma }(x)dz+\frac{1}{\sigma ^{2}}\int_{\Omega }\eta _{\tau
}(x-z)\psi _{\sigma }(z)(p_{\sigma }(z)-p_{\sigma }(x))dz  \label{603}
\end{align}%
and%
\begin{align}
& \int_{\Omega }\eta _{\tau }(x-z)\psi _{\sigma }(z)a(x)dz  \notag \\
=& \int_{\Omega }\eta _{\tau }(x-z)\psi _{\sigma }(z)a(z)dz+\int_{\Omega
}\eta _{\tau }(x-z)\psi _{\sigma }(z)(a(z)-a(x))dz.  \label{604}
\end{align}%
For any $\tau \leq \tau _{0}$, we have%
\begin{equation*}
\eta _{\tau }\ast \big(M_{\sigma ,2,\Omega }[\psi _{\sigma }](x)+(a(x)+\lambda
_{p}(M_{\sigma ,2,\Omega }+a)+\delta )\psi _{\sigma }\big)\geq 0\text{ \ \ for
all }x\in \Omega .
\end{equation*}%
According to (\ref{602}), (\ref{603}) and (\ref{604}), for $\tau \leq \tau _{0}$, we deduce that%
\begin{align*}
M_{\sigma ,2,\Omega }[\widetilde{\varphi }_{\sigma }](x)& +(a(x)+\lambda
_{p}(M_{\sigma ,2,\Omega }+a)+\delta )\widetilde{\varphi }_{\sigma }+\frac{1%
}{\sigma ^{2}}\int_{\Omega }\eta _{\tau }(x-z)\psi _{\sigma }(z)(p_{\sigma
}(z)-p_{\sigma }(x))dz \\
& +\int_{\Omega }\eta _{\tau }(x-z)\psi _{\sigma }(z)(a(z)-a(x))dz\geq 0%
\text{ \ \ \ for all }x\in \Omega .
\end{align*}%
Since $p_{\sigma }(\cdot )$, $a(\cdot )$ are H\"{o}lder continuous, we can
estimate the later two terms of the above inequality by%
\begin{align*}
&\frac{1}{\sigma ^{2}}\bigg|\int_{\Omega }\eta _{\tau }(x-z)\psi _{\sigma
}(z)(p_{\sigma }(z)-p_{\sigma }(x))dz\bigg| \\
\leq &\frac{1}{\sigma ^{2}}\int_{\Omega }\eta _{\tau }(x-z)\psi _{\sigma }(z)%
\frac{|p_{\sigma }(z)-p_{\sigma }(x)|}{|z-x|^{\alpha _{1}}}|z-x|^{\alpha
_{1}}dz \\
\leq &\frac{1}{\sigma ^{2}}K_{1}\tau ^{\alpha _{1}}\widetilde{\varphi }%
_{\sigma }(x)
\end{align*}%
and%
\begin{align*}
&\bigg|\int_{\Omega }\eta _{\tau }(x-z)\psi _{\sigma }(z)(a(z)-a(x))dz\bigg|
\\
\leq &K_{2}\tau ^{\alpha _{2}}\widetilde{\varphi }_{\sigma }(x)
\end{align*}%
where $K_{1}$ and $K_{2}$ are respectively the H\"{o}lder semi-norm of $%
p_{\sigma }$ and $a.$

Thus, for $\tau <\min \Big\{\tau _{0},\big(\frac{\sigma
^{2}\delta }{2K_{1}}\big)^{\frac{1}{\alpha _{1}}}, \big(\frac{\delta }{2K_{2}}\big)^{%
\frac{1}{\alpha _{2}}}\Big\},$ we have%
\begin{equation}
M_{\sigma ,2,\Omega }[\widetilde{\varphi }_{\sigma }](x)+(a(x)+\lambda
_{p}(M_{\sigma ,2,\Omega }+a)+3\delta )\widetilde{\varphi }_{\sigma }(x)\geq
0\text{ \ \ \ \ for all }x\in \Omega .  \label{605}
\end{equation}

Let us consider now $\varphi _{\sigma }:=\gamma \widetilde{\varphi }_{\sigma
},$ where $\gamma $ is a positive constant to be chosen. From (\ref{605}),
we obviously have%
\begin{equation}
M_{\sigma ,2,\Omega }[\varphi _{\sigma }](x)+(a(x)+\lambda _{p}(M_{\sigma
,2,\Omega }+a)+3\delta )\varphi _{\sigma }(x)\geq 0\text{ \ \ \ \ for all }%
x\in \Omega .  \label{606}
\end{equation}%
By taking%
\begin{equation*}
\gamma :=\bigg(\frac{\int_{\Omega }\psi _{\sigma }^{2}(x)dx}{\int_{\Omega }%
\widetilde{\varphi }_{\sigma }^{2}(x)dx}\bigg)^{\frac{1}{2}},
\end{equation*}%
we get%
\begin{equation}
\frac{\int_{\Omega }\psi _{\sigma }^{2}(x)dx}{\int_{\Omega }\varphi _{\sigma
}^{2}(x)dx}=1.  \label{607}
\end{equation}

It is similar to \cite[Lemma 4.5]{Berestycki-2016-JFA}, we obtain%
\begin{equation*}
\lambda^{N}_{1}\Big(\frac{D_{2}(J)}{2N}\Delta +a\Big)\leq \underset{\sigma
\rightarrow 0^{+}}{\lim \inf }~~ \lambda _{p}(M_{\sigma ,2,\Omega }+a)+3\delta .
\end{equation*}
The proof is complete.
\end{proof}
Finally, we point out that the asymptotic behavior of principal eigenfunction of Theorem \ref{th102} ($ii$) is omitted since it is similar to \cite[Theorem 1.5]{Berestycki-2016-JFA}.
To complete the proof of Theorem \ref{th102}, we only need to prove the
asymptotic behavior of principal eigenpair for $m>2$.

\begin{proof}[Proof of Theorem \ref{th102}(iii)]
By Theorem \ref{le307}, for every $\sigma\leq \sigma_{0}$, let $\varphi_{p,\sigma}$ be a positive
eigenfunction associated with $\lambda_{p,\sigma}$, i.e.,
$\lambda_{p,\sigma}$ satisfies
\begin{align}  \label{*}
M_{\sigma,m,\Omega}[\varphi_{p,\sigma}](x)+(a(x)+\lambda_{p,\sigma})
\varphi_{p,\sigma}(x)=0\ \ \ \ \text{for all}\ x\in \Omega.
\end{align}
Let us normalise $\varphi_{p,\sigma}$ by $||\varphi_{p,\sigma}||_{L^{2}(%
\Omega)}=1$.

Set $\{\sigma _{n}\}_{n\in\mathbb{N}}$ with $\sigma _{n}\leq \sigma$ be any sequence of positive reals converging to $0$. Since $||\varphi _{p,\sigma _{n} }||_{L^{2}(\Omega )}\\=1$, there exists a subsequence $\varphi _{p,\sigma _{n} }\rightharpoonup \varphi$ in $%
L^{2}(\Omega )\ ($for convenience, we still denote $\varphi _{p,\sigma _{n} }\rightharpoonup \varphi$).
Let $\rho _{\sigma _{n}}(x-y):=\frac{1}{\sigma _{n}^{2}D_{2}(J)}J_{\sigma
_{n}}(x-y)|x-y|^{2}$.
Multiplying \eqref{*} by $\varphi_{p,\sigma _{n}}$ and integrating over $\Omega$,
we get
\begin{align*}
&\frac{D_{2}(J)}{2}\int_{\Omega}\int_{\Omega}\rho_{\sigma _{n}}(x-y) \frac{|\varphi_{p,\sigma _{n}}(x)-\varphi_{p,\sigma _{n}}(y)|^{2}}{|x-y|^{2}}dxdy \\
=&\sigma _{n}^{m-2}\int_{\Omega}(a(x)+\lambda_{p,\sigma _{n}})\varphi_{p,\sigma _{n}}^{2}dx\leq \sigma^{m-2}\big(\sup_{\Omega}a-\bar{a}\big).
\end{align*}
Thanks to \cite[Theorems 1.2 and 1.3]{Ponce-2004-JEMS},
along with a sequence, we have
\begin{equation}\label{**}
\varphi _{p,\sigma_{n} }\rightarrow \varphi \ \ \text{ in }L^{2}(\Omega ) \ \ and \ \ \varphi\in W^{1,2}.
\end{equation}%
Moreover,
\begin{equation*}
\int_{\Omega}|\nabla\varphi(x)|^{2}dx\leq \sigma^{m-2}(\frac{\sup_{\Omega}a-\bar{a}}{N}).
\end{equation*}%
Since the above inequality holds for every $\sigma$, we obtain
\begin{equation*}
\int_{\Omega}|\nabla\varphi(x)|^{2}dx\leq 0,
\end{equation*}%
which implies that
\begin{equation*}
\varphi(x)\equiv |\Omega|^{-\frac{1}{2}}.
\end{equation*}%
As a consequence, we have
\begin{equation*}
\underset{\sigma\rightarrow 0^{+}}{\lim}\|\varphi_{\sigma,p}-
|\Omega|^{-\frac{1}{2}}\|_{L^{2}(\Omega)}=0
\end{equation*}
and \eqref{**} holds for the entire sequence.

Integrating \eqref{*} over $\Omega$ yields
\begin{equation*}
\int_{\Omega}(a(x)+\lambda_{p,\sigma})\varphi_{p,\sigma}dx=0,
\end{equation*}
i.e.,
\begin{equation*}
\lambda_{p,\sigma}=-\frac{\int_{\Omega}a(x)\varphi_{p,\sigma}dx}
{\int_{\Omega}\varphi_{p,\sigma}dx},
\end{equation*}
which enforces
\begin{equation*}
\lim\limits_{\sigma\to0^{+}}\lambda_p(M_{\sigma,m,\Omega}+a)=
-\bar{a}.
\end{equation*}

\end{proof}

\section{Asymptotic behaviors of positive stationary solution}
\noindent

In this section, we discuss the impact of the dispersal spread and the dispersal budget on the persistence of the population. To do so, we analyze the persistence criteria with respect to the dispersal spread $\sigma$ and the cost parameter $m$. In other words, we investigate the dependence of the positive solution $\theta_{\sigma}$ on $\sigma$ and $m$ for the following equation
\begin{equation}  \label{401}
M_{\sigma,m,\Omega}[\theta]+\theta(a(x)-\theta)=0\ \ \ \ \text{in} \ \Omega.
\end{equation}
To this end, we obtain the existence and uniqueness of the positive solution $\theta_{\sigma}$ to \eqref{401} when $\sigma$ is small or large enough. Furthermore, we analyze the asymptotic behaviors of the positive solution $\theta_{\sigma}$ as $\sigma$ tends to zero or $\infty$ and seek to understand the influence of $m$ on the results.


\subsection{Large dispersal spread}

\noindent

We start by showing some priori estimates for the solutions $\theta_{\sigma}$.

\begin{lemma}
\label{le401} There exist positive constants $C_{1},C_{2}$ such that
for any positive bounded solution $\theta_{\sigma}$ of \eqref{401}, the
following estimates hold
\begin{itemize}
\item[(i)] $||\theta_{\sigma}||_{L^{\infty}(\Omega)}\leq C_{1},\ ||\theta_{\sigma}||^{2}_{L^{2}(\Omega)}\leq C_{2}$\ ;
\item[(ii)] $\theta_{\sigma}\geq \Big(a-\frac{1}{\sigma^{m}}%
\int_{\Omega}J_{\sigma}(x-z)dz\Big)^{+}$.
\end{itemize}
\end{lemma}

\begin{proof}
(i) Since $C_{1}=||a||_{\infty }$ is always a super-solution of \eqref{401},
by the uniqueness of positive solutions of \eqref{401} and comparison principle, $||\theta_{\sigma}||_{L^{\infty}(\Omega)}\leq C_{1}$.
Integrating \eqref{401} over $\Omega $, we get
\begin{equation*}
\int_{\Omega }\theta _{\sigma }^{2}(x)dx=\int_{\Omega }a(x)\theta _{\sigma
}(x)dx\leq \int_{\Omega }a^{+}(x)\theta _{\sigma }(x)dx\leq
C_{1}\int_{\Omega }a^{+}(x)dx:=C_{2}
\end{equation*}%

(ii) Observe that $\Big(a-\frac{1}{\sigma ^{m}}\int_{\Omega }J_{\sigma
}(x-z)dz\Big)^{+}$ is always a sub-solution of \eqref{401}. In fact, we get%
\begin{align*}
&\frac{1}{\sigma ^{m}}\int_{\Omega }J_{\sigma }(x-y)\bigg(a-\frac{1}{\sigma
^{m}}\int_{\Omega }J_{\sigma }(y-z)dz\bigg)^{+}dy \\
&+\bigg(a-\frac{1}{\sigma ^{m}}\int_{\Omega }J_{\sigma }(x-z)dz\bigg)\bigg(a-%
\frac{1}{\sigma ^{m}}\int_{\Omega }J_{\sigma }(x-z)dz\bigg)^{+} \\
&-\bigg[\bigg(a-\frac{1}{\sigma ^{m}}\int_{\Omega }J_{\sigma }(x-z)dz\bigg)%
^{+}\bigg]^{2} \geq 0,
\end{align*}%
which implies that%
\begin{equation*}
\theta _{\sigma }\geq \bigg(a-\frac{1}{\sigma ^{m}}\int_{\Omega
}J_{\sigma}(x-z)dz\bigg)^{+}.
\end{equation*}
The proof is complete.
\end{proof}

Next, we derive an upper bound for $\sigma$ large enough.

\begin{lemma}
\label{le402} There exists $\sigma_{0}>0$ such that for all $m\geq0$ and $%
\sigma\geq\sigma_{0}$, any positive bounded solution $\theta_{\sigma}$ of %
\eqref{401} satisfies
\begin{align*}
\theta_{\sigma}\leq a^{+}+1/\sigma^{\frac{N}{4}}\ \ \ \text{in}\ \Omega.
\end{align*}
\end{lemma}

\begin{proof}
Let $\delta \in (0,N/2)$ and $\zeta _{\sigma }(x):=1/\sigma ^{\frac{N}{2}%
-\delta }+a^{+}(x)$ for $x\in \Omega $. We will show that $\zeta _{\sigma }$
is a super-solution to \eqref{401} when $\sigma $ is large enough. Indeed,
we have for $x\in \Omega $,%
\begin{align*}
&M_{\sigma ,m,\Omega }[\zeta _{\sigma }](x)+\zeta _{\sigma }(x)(a(x)-\zeta
_{\sigma }(x)) \\
=&\frac{1}{\sigma ^{m}}\int_{\Omega }J_{\sigma }(x-y)(\zeta _{\sigma
}(y)-\zeta _{\sigma }(x))dy+\bigg(\frac{1}{\sigma ^{\frac{N}{2}-\delta }}%
+a^{+}(x)\bigg)\bigg(a(x)-a^{+}(x)-\frac{1}{\sigma ^{\frac{N}{2}-\delta }}%
\bigg) \\
=&\frac{1}{\sigma ^{m}}\int_{\Omega }J_{\sigma }(x-y)(a^{+}(y)-a^{+}(x))dy+\bigg(%
\frac{1}{\sigma ^{\frac{N}{2}-\delta }}+a^{+}(x)\bigg)\bigg(a(x)-a^{+}(x)-%
\frac{1}{\sigma ^{\frac{N}{2}-\delta }}\bigg) \\
\leq &\frac{||J||_{L^{\infty }(\mathbb{R}^{N})}}{\sigma ^{N+m}}\int_{\Omega }a^{+}(y)dy-\frac{1}{\sigma ^{m}}%
\int_{\Omega }J_{\sigma }(x-y)dya^{+}(x) \\
&+\bigg(\frac{1}{\sigma ^{\frac{N}{2}-\delta }}+a^{+}(x)\bigg)\bigg(a(x)-a^{+}(x)-%
\frac{1}{\sigma ^{\frac{N}{2}-\delta }}\bigg) \\
\leq &\frac{||J||_{L^{\infty }(\mathbb{R}^{N})}}{\sigma ^{N+m}}\int_{\Omega }a^{+}(y)dy-\frac{1}{\sigma ^{N-2\delta }},
\end{align*}%
where in the last inequality we use%
\begin{equation*}
\bigg(\frac{1}{\sigma ^{\frac{N}{2}-\delta }}+a^{+}(x)\bigg)\bigg(%
a(x)-a^{+}(x)-\frac{1}{\sigma ^{\frac{N}{2}-\delta }}\bigg)\leq -\frac{1}{%
\sigma ^{N-2\delta }}\text{ \ \ for all }x\in \Omega .
\end{equation*}%
Thus, for $\sigma $ large enough, we get for all $x\in \Omega ,$%
\begin{align*}
&M_{\sigma ,m,\Omega }[\zeta _{\sigma }](x)+\zeta _{\sigma }(x)(a(x)-\zeta
_{\sigma }(x)) \\
\leq &\frac{||J||_{L^{\infty }(\mathbb{R} ^{N})}}{\sigma ^{N+m}}\int_{\Omega
}a^{+}(y)dy-\frac{1}{\sigma ^{N-2\delta }} \\
=&\frac{1}{\sigma ^{N-2\delta }}\bigg(\frac{||J||_{L^{\infty }(\mathbb{R}
^{N})}}{\sigma ^{m+2\delta }}\int_{\Omega }a^{+}(y)dy-1\bigg)<0.
\end{align*}%
Therefore, for $\sigma \gg 1,$ we get $\theta _{\sigma }\leq \zeta _{\sigma
}.$ We end the proof by taking $\delta =N/4.$
\end{proof}

Finally, we prove the continuity of $\theta_{\sigma}$ with respect to $%
\sigma $.

\begin{theorem}
\label{le400} Assume that $J$ satisfies $(J)$. Then any positive bounded solution $\theta_{\sigma}$ of \eqref{401} is continuous with
respect to $\sigma$.
\end{theorem}

\begin{proof}
Define
$
G(\sigma ,\theta ):=M_{\sigma ,m,\Omega }[\theta ]+\theta (a(x)-\theta ).
$
For any given $\sigma _{0}>0$, there exists $\theta _{\sigma _{0}}\in C(\bar{%
\Omega})$ such that
$
G(\sigma _{0},\theta _{\sigma _{0}})=M_{\sigma _{0},m,\Omega }[\theta
_{\sigma _{0}}]+\theta _{\sigma _{0}}(a(x)-\theta _{\sigma _{0}}).
$
Note that for any $v\in C(\bar{\Omega})$
\begin{equation*}
G_{\theta }(\sigma _{0},\theta _{\sigma _{0}})v=M_{\sigma _{0},m,\Omega
}[v]+(a(x)-2\theta _{\sigma _{0}})v.
\end{equation*}%
Define
$
L_{1}[v](x):=M_{\sigma _{0},m,\Omega }[v]+(a(x)-\theta _{\sigma _{0}})v \text{ and }
L_{2}[v](x):=M_{\sigma _{0},m,\Omega }[v]+(a(x)-2\theta _{\sigma _{0}})v.
$
We know that 0 is the principal eigenvalue of $L_{1}$ with eigenfunction $%
\theta _{\sigma _{0}}>0$. Thus, it follows from \cite[lemma 2.2]%
{Bates-2007-JMAA} that
\begin{equation*}
\sup \{Re\lambda |\lambda \in \sigma (L_{1})\}=0,
\end{equation*}%
where $\sigma (L_{1})$ is the spectrum of $L_{1}$. By the definition of $%
L_{1}$ and $L_{2}$, we have
\begin{equation*}
\sup \{Re\lambda |\lambda \in \sigma (L_{2})\}<\sup \{Re\lambda |\lambda \in
\sigma (L_{1})\}=0.
\end{equation*}%
This implies that $0\in \rho (L_{2})$, where $\rho (L_{2})=\mathbb{C}\backslash \sigma (L_{2})$ is the resolvent
set of $L_{2}$. By the definition of $\rho (L_{2})$, we have $G_{\theta
}^{-1}(\sigma _{0},\theta _{\sigma _{0}})$ is a linear continuous operator
in $C(\bar{\Omega})$. Then by the Implicit Function Theorem, $\theta
_{\sigma }$ is a continuous function with respect to $\sigma $.
\end{proof}


\noindent

Now, we analyze the situation that $\sigma$ is large enough.
\begin{proof}[Proof of Theorem \ref{th103}]
It follows from Theorem \ref{th101} that we get $ \underset{\sigma\rightarrow \infty}{\lim}\lambda_{p}(M_{\sigma,m,\Omega}+a)=
-\underset{\Omega}{\sup}\ a
$
and $\lambda_{p}(M_{\sigma,m,\Omega}+a)$ is continuous with respect to $%
\sigma$. As a result, there exists $\sigma_{1}>0$ such that for all $%
\sigma\geq \sigma_{1}$, we obtain
\begin{align*}
\lambda_{p}(M_{\sigma,m,\Omega}+a)\leq -\frac{1}{2}\underset{\Omega}{\sup} \
a<0.
\end{align*}
By Lemma \ref{th207}, for all $\sigma\geq \sigma_{1}$, there exists a
positive solution $\theta_{\sigma}$ to \eqref{401}.

Owing to Lemma \ref{le401} ($ii$) and Lemma \ref{le402}, for all $x\in \Omega$
and large $\sigma$, we have
\begin{align*}
\bigg(a-\frac{1}{\sigma^{m}}\int_{\Omega}J_{\sigma}(x-y)dy\bigg)^{+} \leq
\theta_{\sigma}\leq a^{+}+1/\sigma^{\frac{N}{4}}.
\end{align*}
Since
\begin{align*}
\int_{\Omega}J_{\sigma}(x-y)dy=\int_{\frac{\Omega-x}{\sigma}}J(z)dz
\rightarrow 0 \ \ \ \ \text{as}\ \sigma\rightarrow \infty,
\end{align*}
$\theta_{\sigma}$ converges uniformly to $a^{+}$.
\end{proof}

\subsection{Small dispersal spread}

\noindent

Firstly, we consider the situation $0\leq m<2$. Although the proof of Theorem \ref{th104} (i) is similar to \cite[Theorem 1.2]{Berestycki-2016-JMB}, the details are a little different.
For the convenience of the reader, we will give the proof.
\begin{proof}[Proof of Theorem \ref{th104} (i)]
It follows from Theorem \ref{th101} that we know
\begin{align*}
\underset{\sigma\rightarrow 0^{+}}{\lim}\lambda_{p}(M_{\sigma,m,\Omega}+a)= -%
\underset{\Omega}{\sup}\ a
\end{align*}
and $\lambda_{p}(M_{\sigma,m,\Omega}+a)$ is continuous with respect to $%
\sigma$. As a consequence, there exists $\sigma_{0}>0$ such that for all $%
\sigma\leq \sigma_{0}$, we have
\begin{align*}
\lambda_{p}(M_{\sigma,m,\Omega}+a)\leq -\frac{1}{2}\underset{\Omega}{\sup} \
a<0.
\end{align*}
By Lemma \ref{th207}, for all $\sigma\leq \sigma_{0}$, there exists a
positive solution $\theta_{\sigma}$ to \eqref{401}.

In what follows, let us determine the limit of $\theta _{\sigma }$ as $%
\sigma \rightarrow 0^{+}$. Let $w_{\sigma }:=a-\theta _{\sigma }$, then from %
\eqref{401}, $w_{\sigma }$ satisfies
\begin{equation*}
-\frac{1}{\sigma ^{m}}\int_{\Omega }J_{\sigma }(x-y)(w_{\sigma
}(y)-w_{\sigma }(x))dy+\theta _{\sigma }w_{\sigma }=-\frac{1}{\sigma ^{m}}%
\int_{\Omega }J_{\sigma }(x-y)(a(y)-a(x))dy\ \ \ \ \text{in}\ \Omega .
\end{equation*}%
Multiplying this equation by $w_{\sigma }^{+}$ with $w_{\sigma }^{+}:=\sup
\{w_{\sigma },0\}$ and integrating over $\Omega $, it follows that
\begin{align}
&\frac{1}{\sigma ^{m}}\int_{\Omega }\int_{\Omega }J_{\sigma
}(x-y)[(w_{\sigma }^{+})^{2}(x)-w_{\sigma }(y)w_{\sigma
}^{+}(x)]dxdy+\int_{\Omega }\theta _{\sigma }(x)(w_{\sigma }^{+})^{2}(x)dx
\notag \\
=&\int_{\Omega }w_{\sigma }^{+}(x)g_{\sigma }(x)dx,  \label{402}
\end{align}%
with $g_{\sigma }(x):=-\frac{1}{\sigma ^{m}}\int_{\Omega }J_{\sigma
}(x-y)(a(y)-a(x))dy$.

Let us now estimate the above integrals. Indeed, since $w_{\sigma
}(y)=w_{\sigma }^{+}(y)-w_{\sigma }^{-}(y)$ with $w_{\sigma }^{-}(y)=\sup
\{-w_{\sigma },0\}$, we get%
\begin{align}
& \frac{1}{\sigma ^{m}}\int_{\Omega }\int_{\Omega }J_{\sigma }(x-y)\big[%
(w_{\sigma }^{+})^{2}(x)-w_{\sigma }(y)w_{\sigma }^{+}(x)\big]dxdy  \notag \\
=& \frac{1}{\sigma ^{m}}\int_{\Omega }\int_{\Omega }J_{\sigma }(x-y)\big[%
(w_{\sigma }^{+})^{2}(x)-w_{\sigma }^{+}(y)w_{\sigma }^{+}(x)\big]dxdy
\notag \\
& +\frac{1}{\sigma ^{m}}\int_{\Omega }\int_{\Omega }J_{\sigma
}(x-y)w_{\sigma }^{+}(x)w_{\sigma }^{-}(y)dxdy  \notag \\
=& \frac{1}{2\sigma ^{m}}\int_{\Omega }\int_{\Omega }J_{\sigma }(x-y)\big[%
w_{\sigma }^{+}(x)-w_{\sigma }^{+}(y)\big]^{2}dxdy+\frac{1}{\sigma ^{m}}%
\int_{\Omega }\int_{\Omega }J_{\sigma }(x-y)w_{\sigma }^{+}(x)w_{\sigma
}^{-}(y)dxdy.  \label{403}
\end{align}%
By Lemma \ref{le401} (i), we get%
\begin{equation*}
\bigg|\int_{\Omega }w_{\sigma }^{+}(x)g_{\sigma }(x)dx\bigg|\leq
2C_{1}\int_{\Omega }|g_{\sigma }(x)|dx.
\end{equation*}%
Let us estimate $g_{\sigma }(x)$ for all $x\in \Omega ,$%
\begin{align*}
g_{\sigma }(x) =&-\frac{1}{\sigma ^{m}}\int_{\Omega }J_{\sigma
}(x-y)(a(y)-a(x))dy \\
=&-\frac{1}{\sigma ^{m}}\int_{\frac{\Omega -x}{\sigma }}J(z)(a(x+\sigma
z)-a(x))dz.
\end{align*}%
For enough small $\sigma $, 
we obtain $supp(J)=B_{1}(0)\subset \frac{\Omega -x}{\sigma }$. Since $a\in
C^{2}(\overline{\Omega })$, a Taylor expansion leads to%
\begin{align*}
g_{\sigma }(x) =&-\frac{1}{\sigma ^{m}}\int_{\mathbb{R}^{N}}J(z)(a(x+\sigma z)-a(x))dz\\
=&-\frac{1}{\sigma ^{m}}\int_{B_{1}(0)}J(z)(a(x+\sigma z)-a(x))dz \\
=&-\frac{1}{\sigma ^{m}}\int_{B_{1}(0)}J(z)\Big(\nabla a(x)\cdot \sigma z+%
\frac{\sigma ^{2}}{2}z^{T}D^{2}a(x)z+o(\sigma ^{2})\Big)dz.
\end{align*}%
Thus, for all $x\in \Omega $, we obtain%
\begin{equation*}
|g_{\sigma }(x)|\leq \frac{D_{2}(J)}{2N}||\triangle a||_{L^{\infty }(\Omega
)}\sigma ^{2-m}+o(\sigma ^{2-m})
\end{equation*}%
and%
\begin{equation}
\Bigg|\int_{\Omega }w_{\sigma }^{+}(x)g_{\sigma }(x)dx\Bigg|\leq 2C_{1}|\Omega
|\bigg(\frac{D_{2}(J)}{2N}||\triangle a||_{L^{\infty }(\Omega )}+o(1)\bigg)\sigma ^{2-m}.
\label{404}
\end{equation}%
Collecting (\ref{402}), (\ref{403}) and (\ref{404}), we have%
\begin{align*}
&\frac{1}{2\sigma ^{m}}\int_{\Omega }\int_{\Omega }J_{\sigma }(x-y)\big[%
w_{\sigma }^{+}(x)-w_{\sigma }^{+}(y)\big]^{2}dxdy+\frac{1}{\sigma ^{m}}%
\int_{\Omega }\int_{\Omega }J_{\sigma }(x-y)w_{\sigma }^{+}(x)w_{\sigma
}^{-}(y)dxdy \\
&+\int_{\Omega }\theta _{\sigma }(x)(w_{\sigma }^{+})^{2}(x)dx \leq C\sigma ^{2-m}.
\end{align*}%
Thus%
\begin{equation*}
\int_{\Omega }\theta _{\sigma }(x)(w_{\sigma }^{+})^{2}(x)dx\leq C\sigma
^{2-m}
\end{equation*}%
and%
\begin{equation}
\theta _{\sigma }(x)w_{\sigma }^{+}(x)\rightarrow 0\text{ \ \ a.e. in }\Omega \ \ \text{ as }\sigma \rightarrow 0^{+}.  \label{405}
\end{equation}%

Let us denote $Q:=supp(w_{\sigma }^{+})$, integrating (\ref{401}) over  $%
\Omega $ yields
\begin{equation*}
\int_{\Omega }\theta _{\sigma }(a(x)-\theta _{\sigma })dx=0.
\end{equation*}%
We conclude from (\ref{405}) that
\begin{align*}
\int_{\Omega \backslash Q}\theta _{\sigma }(a(x)-\theta _{\sigma })dx
=-\int_{Q}\theta _{\sigma }(a(x)-\theta _{\sigma })dx
=-\int_{Q}\theta _{\sigma }w_{\sigma }dx\rightarrow 0\text{ \ as }\sigma
\rightarrow 0^{+}.
\end{align*}%
Since $\theta _{\sigma }(a(x)-\theta _{\sigma })\leq 0$ in $\Omega
\backslash Q$, it follows that%
\begin{equation} \label{406}
\theta _{\sigma }(x)w_{\sigma }(x)\rightarrow 0\text{ \ \ a.e. in }\Omega \backslash Q \ \ \text{ as }\sigma \rightarrow 0^{+}.
\end{equation}%
Collecting (\ref{405}) and (\ref{406}), we have
$
\theta _{\sigma }\rightarrow V_{1} \text{a.e. in }\Omega \text{
as }\sigma \rightarrow 0^{+},
$
where $V_{1}$ is a nonnegative bounded solution of $V_{1}\big(a(x)-V_{1}\big)=0~ \text{ in }\Omega.
$
This ends the proof.
\end{proof}

Next, we discuss the case $m\geq2$. The proof of Theorem \ref{th104} (ii) is similar to \cite[Theorem 1.4]{Berestycki-2016-JMB}. Here, we omit it. Therefore, to complete the proof of Theorem \ref{th104}, we only need to consider the case $m>2$.

\begin{proof}[Proof of Theorem \ref{th104} (iii)]
Assume for the moment that $\bar{a}>0$. By Theorem \ref{th102}, there exists $\sigma _{0}>0$ such that for
all $\sigma \leq \sigma _{0}$, we have
\begin{equation*}
\lambda _{p}(M_{\sigma ,m,\Omega })\leq -\frac{1}{2}\bar{a}<0.
\end{equation*}%
Thanks to Lemma \ref{th207}, for all $\sigma \leq \sigma _{0}$, there is a positive solution $\theta _{\sigma }$ of \eqref{401}.

For every $\sigma\leq \sigma_{0}$,
$\theta _{\sigma }$ satisfies
\begin{align}  \label{***}
M_{\sigma,m,\Omega}[\theta _{\sigma }](x)+\theta _{\sigma } (a(x)-\theta _{\sigma })=0\ \ \ \ \text{for all}\ x\in \Omega.
\end{align}
Set $\{\sigma _{n}\}_{n\in\mathbb{N}}$ with $\sigma _{n}\leq \sigma$ be any sequence of positive reals converging to $0$. We write $\theta _{n}$ instead of $\theta _{\sigma _{n}}$.
Multiplying \eqref{***} by $\theta _{n}$ and integrating over $\Omega$ yield
\begin{equation*}
\frac{1}{2\sigma _{n}^{m}}\int_{\Omega }\int_{\Omega }J_{\sigma
_{n}}(x-y)(\theta _{n}(y)-\theta _{n}(x))^{2}dydx=\int_{\Omega }\theta
_{n}^{2}(x)(a(x)-\theta _{n}(x))dx.
\end{equation*}%
Let $\rho _{\sigma _{n}}(x-y):=\frac{1}{\sigma _{n}^{2}D_{2}(J)}J_{\sigma
_{n}}(x-y)|x-y|^{2}$, we get%
\begin{equation*}
\frac{D_{2}(J)}{2}\int_{\Omega }\int_{\Omega }\rho _{\sigma _{n}}(x-y)\frac{%
(\theta _{n}(y)-\theta _{n}(x))^{2}}{|x-y|^{2}}dxdy=\sigma _{n}^{m-2}\int_{\Omega }\theta
_{n}^{2}(x)(a(x)-\theta _{n}(x))dx.
\end{equation*}%
By Lemma \ref{le401}($i$), it follows that%
\begin{equation*}
\frac{D_{2}(J)}{2}\int_{\Omega }\int_{\Omega }\rho _{n}(x-y)\frac{(\theta
_{n}(y)-\theta _{n}(x))^{2}}{|x-y|^{2}}dxdy\leq C\sigma^{m-2}
\end{equation*}%
with $C$ independent of $n.$

Since $||\theta _{n}||_{L^{2}(\Omega )}$ is uniformly bounded by Lemma \ref{le401}($i$), there exists a subsequence $\theta _{n}\rightharpoonup \psi$ in $L^{2}(\Omega )\ ($for convenience, we still denote $\{\theta _{n}\}_{n\in\mathbb{N}})$. Thanks to \cite[Theorems 1.2 and 1.3]{Ponce-2004-JEMS}, along a sequence, we have
\begin{equation}\label{****}
\theta _{n}\rightarrow \psi \ \ \text{ in }L^{2}(\Omega ) \ \ and \ \ \psi\in W^{1,2}.
\end{equation}%
Moreover,
\begin{equation*}
\int_{\Omega}|\nabla\psi(x)|^{2}dx\leq \frac{C\sigma^{m-2}}{N}.
\end{equation*}%

Since the above inequality holds for every $\sigma$, there holds
\begin{equation*}
\int_{\Omega}|\nabla\psi(x)|^{2}dx\leq 0,
\end{equation*}%
which implies that
\begin{equation}\label{*****}
\theta _{n}\rightarrow K \ \ \text{ in }L^{2}(\Omega )
\end{equation}%
with the constant $K$ and $\psi\equiv K$.

To conclude we need to prove that $K$ is non-zero.
Now, we give a brief explanation.
It follows from Theorem \ref{th102} (iii) that there exists a positive principal
eigenfunction $\varphi _{p,\sigma }$ associated with $\lambda _{p}(M_{\sigma
,m,\Omega }+a)$ normalised by $||\varphi _{p,\sigma }||_{L^{2}(\Omega )}=1$.
Thus, we obtain%
\begin{equation*}
M_{\sigma ,m,\Omega }[\varphi _{p,\sigma }]+(a(x)+\lambda _{p}(M_{\sigma
,m,\Omega }+a))\varphi _{p,\sigma }=0\text{ \ \ \ in }\Omega ,
\end{equation*}%
which implies that%
\begin{equation}
M_{\sigma ,2,\Omega }[\varphi _{p,\sigma }]+a(x)\varphi _{p,\sigma }\geq
\frac{1}{2}\bar{a}\varphi
_{p,\sigma }.  \label{411}
\end{equation}%
By Theorem \ref{th102}(iii), there holds%
\begin{equation}
\varphi _{p,\sigma }\rightarrow |\Omega|^{-\frac{1}{2}}\text{ \ \ \ \ a.e. in }\Omega~~~\text{as}~~~\sigma\to0^+,
\label{412}
\end{equation}%
Then there exists $\alpha >0$ such that
$
\alpha \varphi _{p,\sigma }\rightarrow \alpha |\Omega|^{-\frac{1}{2}}<1/2\text{ a.e. in }\Omega~\text{as}~\sigma\to0^+ .
$
Plugging $\alpha \beta \varphi _{p,\sigma }$ in (\ref{401}), it follows from
(\ref{411}) that%
\begin{align*}
M_{\sigma ,m,\Omega }[\alpha \beta \varphi _{p,\sigma }]+\alpha \beta
\varphi _{p,\sigma }(a(x)-\alpha \beta \varphi _{p,\sigma })
=&\alpha \beta \big(M_{\sigma ,m,\Omega }[\varphi _{p,\sigma }]+a(x)\varphi
_{p,\sigma }\big)-\alpha ^{2}\beta ^{2}\varphi _{p,\sigma }^{2} \\
\geq &\frac{\bar{a}}{2}\alpha \beta \varphi _{p,\sigma }-\alpha ^{2}\beta ^{2}\varphi
_{p,\sigma }^{2} =\alpha \beta \varphi _{p,\sigma }\Big(\frac{\bar{a}}{2}-\alpha \beta \varphi _{p,\sigma }\Big) \\
>&\alpha \beta \varphi _{p,\sigma }\Big(\frac{\bar{a}}{2}-\beta \Big).
\end{align*}%
Let $\beta =\frac{\bar{a}}{4}$,
the function $\frac{\alpha\bar{a} }{4}\varphi _{p,\sigma }$ is a sub-solution to (\ref{401}), i.e., for some $%
\sigma _{1}>0,$%
\begin{equation*}
\theta _{\sigma }\geq \frac{\alpha\bar{a} }{4}\varphi _{p,\sigma }\text{ \ \ a.e. in }\Omega \text{, for all
}\sigma \leq \sigma _{1},
\end{equation*}%
which combined with (\ref{412}) enforces for some $\gamma ,\sigma _{2}>0,$%
\begin{equation*}
\theta _{\sigma }\geq \gamma |\Omega|^{-\frac{1}{2}}\text{ \ \ a.e. in }\Omega \text{,
for all }\sigma \leq \sigma _{2}.
\end{equation*}%
Thus, we have $K>0$.

On the other hand, integrating \eqref{***} over $\Omega$ yields
$
\int_{\Omega}\theta_{n}(a(x)-\theta_{n})dx=0.
$
As $n\rightarrow +\infty$, there holds
$
\int_{\Omega}K(a(x)-K)dx=0,
$
which implies that
\begin{equation*}
K=\frac{1}{|\Omega|}\int_{\Omega}a(x)dx.
\end{equation*}
As a consequence, we have
$
\underset{\sigma\rightarrow 0^{+}}{\lim}\|\theta_{\sigma}-\bar{a}\|_{L^{2}(\Omega)}=0
$
and \eqref{*****} holds for the entire sequence.

\end{proof}

\section{Asymptotic properties of solution $u_{\sigma}(x,t)$}
\noindent

This section is devoted to studying the following nonlocal dispersal evolution equation
\begin{equation}  \label{800}
\begin{cases}
u_{t}(x,t)=\frac{1}{\sigma^{m}}\int_{\Omega}J_{\sigma}(x-y)(u(y,t)-u(x,t))
dy+u(a(x)-u),\ \  &(x,t)\in \bar{\Omega}\times(0,T],\\
u(x,0)=u_{0}(x),\ \ &x\in \bar{\Omega}.
\end{cases}
\end{equation}
More precisely, we are concerned about asymptotic of solution $u_{\sigma}(x,t)$ with respect to the dispersal spread $\sigma$.
Here, we recall the existence and uniqueness of solution
$u_{\sigma}(x,t)$ to \eqref{800}.
Furthermore, we analyze the asymptotic behaviors of solution $u_{\sigma}(x,t)$ and give the convergence rate as $\sigma$ tends to zero or $\infty$.

\subsection{Large dispersal spread}
\noindent

We first recall the definition of sub-solution (super-solution) to \eqref{800} and give a known result, see \cite{Bates-2007-JMAA,Rossi-2007-JDE}.
\begin{definition}\label{de401}
Let $X=C(\bar{\Omega})$, $S_{T}=\bar{\Omega}\times(0,T)$ for $0<T\leq \infty$. We say
$u\in C^{1}([0,T), X)$ is a sub-solution (super-solution) of \eqref{800} in $S_{T}$ if
\begin{equation*}
u_{t}\leq (\geq)~\frac{1}{\sigma^{m}}\int_{\Omega}J_{\sigma}(x-y)(u(y,t)-u(x,t))
dy+u(a(x)-u).
\end{equation*}
\end{definition}
\begin{proposition}\label{pr401}
Assume that $J$ satisfies $(J)$ and $a\in C(\bar{\Omega})$. Then \eqref{800} has a global solution $u(x,\cdot;u_{0})$ for some continuous and non-negative initial function $u_{0}$.
\end{proposition}

Next, an ordinary differential equation result is given. Since it is easy checked, we omit its proof.

\begin{proposition}\label{pr402}
Assume that $a\in C(\bar{\Omega})$ and $u_{0}\in C(\bar{\Omega})$ with
$u_{0}\geq 0$. Then, for each $x\in \Omega$, the equation
\begin{equation}\label{4111}
\begin{cases}
v_{t}(x,t)=v(a(x)-v) \ \ &(x,t)\in \bar{\Omega}\times(0,\infty),\\
v_{t}(x,0)=u_{0}(x)  \ \ &x\in \bar{\Omega},
\end{cases}
\end{equation}
has a unique solution $v(\cdot,t;u_{0})$ that is continuous in $x$.
Moreover, if $a, u_{0}\in C^{2}(\bar{\Omega})$, then $v(\cdot,t;u_{0})\in C^{2}(\bar{\Omega})$ for every $t\geq0$.
\end{proposition}

Now, let us prove the main result.
\begin{proof}[Proof of Theorem \ref{th106}]
Let $w_{\sigma}=u_{\sigma}-v$ with $v$ a solution of \eqref{4111}, then there holds
\begin{align}\label{801}
\begin{cases}
(w_{\sigma})_{t}(x,t)= M_{\sigma,m,\Omega}[w_{\sigma}](x,t)+a_{\sigma}(x,t)
w_{\sigma}+F_{\sigma}(x,t)\ \ \ &\text{in } \bar{\Omega}\times(0,T],\\
w_{\sigma}(x,0)=0 &\text{on } \bar{\Omega},
\end{cases}
\end{align}
where $a_{\sigma}(x,t)=a(x)+u_{\sigma}+v$, $F_{\sigma}(x,t)=M_{\sigma,m,\Omega}[v](x,t)$.

Take $K=\max\{\sup_{\Omega}a, \sup_{\Omega}u_{0}\}$. By comparison principle, it is easy to verify that
\begin{equation*}
\sup\limits_{t\in[0,T]}\|u_{\sigma}(\cdot,t)\|_{L^{\infty}(\Omega)}\leq K, \ \ \sup\limits_{t\in[0,T]}\|v(\cdot,t)\|_{L^{\infty}(\Omega)}\leq K.
\end{equation*}
Thus, there hold
\begin{equation*}
\sup\limits_{t\in[0,T]}\|a_{\sigma}(\cdot,t)\|_{L^{\infty}(\Omega)}\leq 3K
\end{equation*}
and
\begin{align}\label{802}
\sup\limits_{t\in[0,T]}\|F_{\sigma}(\cdot,t)\|_{L^{\infty}(\Omega)}\leq \frac{2K}{\sigma^{m}}\sup\limits_{x\in\Omega}\bigg|\int_{\Omega}J_{\sigma}
(x-y)dy\bigg|=\frac{2K}{\sigma^{m}}\sup\limits_{x\in\Omega}\bigg|\int_{\frac{\Omega-x}
{\sigma}}J(z)dz\bigg|\leq O(\sigma^{-(m+N)}).
\end{align}

Next, let $\bar{w}$ be given by
\begin{equation*}
\bar{w}(x,t)=e^{3Kt}(K_{1}\sigma^{-(m+N)}t).
\end{equation*}
By direct calculation, there holds
\begin{align}\label{803}
\begin{cases}
(\bar{w})_{t}(x,t)= M_{\sigma,m,\Omega}[\bar{w}](x,t)+a_{\sigma}
\bar{w}+\bar{F}_{\sigma}(x,t)\ \ \ &\text{in } \bar{\Omega}\times(0,T],\\
\bar{w}(x,0)=0 &\text{on }\bar{\Omega},
\end{cases}
\end{align}
where $\bar{F}_{\sigma}(x,t)=e^{3Kt}K_{1}\sigma^{-(m+N)}+
(3K-a_{\sigma}(x,t))\bar{w}$.
By \eqref{802}, there are $\sigma_{2}>0$ and $K_{1}>0$ such that
\begin{equation}\label{804}
\bar{F}_{\sigma}(x,t)>F_{\sigma}(x,t)\ \ \text{for all }\sigma\geq\sigma_{2}.
\end{equation}

Due to \eqref{801}, \eqref{803}, \eqref{804} and comparison principle,
we obtain
\begin{equation}\label{805}
w_{\sigma}(x,t)\leq\bar{w}_{\sigma}(x,t)\ \ \ \forall(x,t)\in \bar{\Omega}\times[0,T],
\end{equation}
for all $\sigma\geq\sigma_{2}$.

Similarly, let $\underline{w}(x,t)=-e^{3Kt}(K_{1}\sigma^{-(m+N)}t)$. We can prove that there exists $\sigma_{3}>0$ such that
\begin{equation}\label{806}
w_{\sigma}(x,t)\geq\underline{w}_{\sigma}(x,t)\ \ \ \forall(x,t)\in \bar{\Omega}\times[0,T],
\end{equation}
for all $\sigma\geq\sigma_{3}$.

By \eqref{805} and \eqref{806}, we obtain
\begin{equation}
|w_{\sigma}(x,t)|\leq e^{3Kt}(K_{1}\sigma^{-(m+N)}t)\ \ \ \forall(x,t)\in \bar{\Omega}\times[0,T],
\end{equation}
which implies that there is $C(T)>0$ such that
\begin{equation*}
\sup\limits_{t\in[0,T]}\|u_{\sigma}(\cdot,t)-v(\cdot,t)\|_{L^{\infty}
(\Omega)}\leq C(T)\sigma^{-(m+N)}
\end{equation*}
for all $\sigma\geq \sigma_{1}:=\max\{\sigma_{2},\sigma_{3}\}$.
The proof is complete.
\end{proof}

\subsection{Small dispersal spread}
\noindent

Finally, let us discuss the asymptotic behavior of solution $u_{\sigma}$ for small dispersal spread.
\begin{proof}[Proof of Theorem \ref{th107}]
Let $w_{\sigma}=u_{\sigma}-v$ with $v$ a solution of \eqref{4111}, then there holds
\begin{align}\label{807}
\begin{cases}
(w_{\sigma})_{t}(x,t)= M_{\sigma,m,\Omega}[w_{\sigma}](x,t)+a_{\sigma}(x,t)
w_{\sigma}+F_{\sigma}(x,t)\ \ \ &\text{in } \bar{\Omega}\times(0,T),\\
w_{\sigma}(x,0)=0 &\text{on }\bar{\Omega},
\end{cases}
\end{align}
where $a_{\sigma}(x,t)=a(x)+u_{\sigma}+v$, $F_{\sigma}(x,t)=M_{\sigma,m,\Omega}[v](x,t)$.

In fact, for $\sigma$ small enough, we have $B_{1}(0)\subset \frac{\Omega-x}{\sigma}$ and
\begin{align*}
F_{\sigma}(x,t)&=\frac{1}{\sigma^{m}}\int_{\Omega}J_{\sigma}
(x-y)(v(y,t)-v(x,t))dy
=\frac{1}{\sigma^{m}}\int_{\frac{\Omega-x}{\sigma}}J(z)(v(x+\sigma z,t)-v(x,t))dz\\
&=\frac{1}{\sigma^{m}}\int_{\mathbb{R}^N}J(z)(v(x+\sigma z,t)-v(x,t))dz
=\sigma^{2-m}\int_{\mathbb{R}^N}J(z)\bigg(\frac{z^{T}D^{2}v(x,t)z}{2}+
O(1)\bigg)dz\\
&= O(\sigma^{2-m}),\ \ \text{for all} \ x\in\Omega.
\end{align*}
Now, let $\bar{w},\underline{w}$ be given by
\begin{equation*}
\bar{w}(x,t)=e^{3Kt}(K_{1}\sigma^{2-m}t), \ \
\underline{w}(x,t)=-e^{3Kt}(K_{1}\sigma^{2-m}t)
\end{equation*}
By direct calculation, there exist $\sigma_{0}>0$ and $K_{1}>0$ such that
\begin{equation*}
\underline{w}(x,t)\leq w_{\sigma}(x,t)\leq \bar{w}_{\sigma}(x,t)\ \ \ \forall(x,t)\in \bar{\Omega}\times[0,T],
\end{equation*}
for all $\sigma\leq\sigma_{0}$.

Thus, we obtain
\begin{equation}
|w_{\sigma}(x,t)|\leq e^{3Kt}(K_{1}\sigma^{2-m}t)\ \ \ \forall(x,t)\in \bar{\Omega}\times[0,T],
\end{equation}
which implies that there is $C(T)>0$ such that
\begin{equation*}
\sup\limits_{t\in[0,T]}\|u_{\sigma}(\cdot,t)-v(\cdot,t)\|_{L^{\infty}
(\Omega)}\leq C(T)\sigma^{2-m}
\end{equation*}
for all $0<\sigma\leq\sigma_{0}$.
This ends the proof.
\end{proof}

\section*{Acknowledgments}
\noindent

The authors would like to thank the referees for their careful reading and valuable suggestions. The second author was partially supported by NSF of China (11731005, 11671180) and the third author was partially supported by NSF of China (11601205).


\end{document}